\documentclass{amsart}

\usepackage{amsfonts, amssymb, amsmath, eucal, verbatim, amsthm, amscd, enumerate}

\usepackage{graphicx}
\usepackage{framed}
\usepackage{tikz}
\usepackage{enumitem}

\usepackage{ascmac}

\DeclareMathOperator*{\esssup}{ess\,sup}

\newtheorem{theorem}{Theorem}[section]
\newtheorem{lemma}[theorem]{Lemma}
\newtheorem{proposition}[theorem]{Proposition}

\newtheorem{corollary}[theorem]{Corollary}
\theoremstyle{definition}
\newtheorem{definition}[theorem]{Definition}

\newtheorem{claim}[theorem]{Claim}

\theoremstyle{remark}
\newtheorem*{remark}{Remark}
\newtheorem*{remarks}{Remarks}

\newtheorem*{organization}{Organization}

\newcommand{\vertiii}[1]{{\left\vert\kern-0.25ex\left\vert\kern-0.25ex\left\vert #1 
\right\vert\kern-0.25ex\right\vert\kern-0.25ex\right\vert}}

\numberwithin{equation}{section}

\usepackage{setspace}

\def\1{\textbf{\rm 1}}

\def\XXint#1#2#3{{\setbox0=\hbox{$#1{#2#3}{\int}$}
\vcenter{\hbox{$#2#3$}}\kern-.5\wd0}}

\usepackage[margin=40mm]{geometry}

\begin{document}

\date{\today}

\author[Bez]{Neal Bez}
\address[Neal Bez]{Department of Mathematics, Graduate School of Science and Engineering,
Saitama University, Saitama 338-8570, Japan}
\email{nealbez@mail.saitama-u.ac.jp}
\author[Nakamura]{Shohei Nakamura}
\address[Shohei Nakamura]{Department of Mathematics, Graduate School of Science, Osaka University, Toyonaka, Osaka 560-0043, Japan}
\email{srmkn@math.sci.osaka-u.ac.jp}
\thanks{This work was supported by JSPS Kakenhi grant numbers 18KK0073, 19H00644, 19H01796 and 23H01080 (Bez), and Grant-in-Aid for JSPS Research Fellow no. 17J01766 and JSPS Kakenhi grant numbers 19K03546, 19H01796 and 21K13806 (Nakamura).
}

\title[Regularized Brascamp--Lieb inequalities]{Regularized Brascamp--Lieb inequalities}

\begin{abstract}
Given any (forward) Brascamp--Lieb inequality on euclidean space, a famous theorem of Lieb guarantees that gaussian near-maximizers always exist.
Recently, Barthe and Wolff used mass transportation techniques to establish a counterpart to Lieb's theorem for all non-degenerate cases of the inverse Brascamp--Lieb inequality. Here we build on work of Chen--Dafnis--Paouris and employ heat-flow techniques to understand the inverse Brascamp--Lieb inequality for certain regularized input functions, in particular extending the Barthe--Wolff theorem to such a setting. Inspiration arose from work of Bennett, Carbery, Christ and Tao for the forward inequality, and we recover their generalized Lieb's theorem using a clever limiting argument of Wolff. In fact, we use Wolff's idea to deduce regularized inequalites in the broader framework of the forward-reverse Brascamp--Lieb inequality, in particular allowing us to recover the gaussian saturation property in this framework first obtained by Courtade, Cuff, Liu and Verd\'u.
\end{abstract}

\maketitle

\section{Introduction} \label{section:intro}
The main subject of this paper is the inverse Brascamp--Lieb inequality 
\begin{equation}\label{e:IBL}
\int_{\mathbb{R}^n} e^{-\pi \langle x, \mathcal{Q}x\rangle}  \prod_{j=1}^m f_j(B_jx)^{c_j}\, dx \ge C  \prod_{j=1}^m \bigg( \int_{\mathbb{R}^{n_j}} f_j \bigg)^{c_j}
\end{equation}
associated with a self-adjoint linear transformation $\mathcal{Q} : \mathbb{R}^n \to \mathbb{R}^n$, and families of linear transformations $B_j : \mathbb{R}^n \to \mathbb{R}^{n_j}$ and exponents $c_j \in \mathbb{R}$. Largely thanks to recent work of Barthe--Wolff \cite{BWAnn,BW} and Chen--Dafnis--Paouris \cite{CDP}, much is known about inverse Brascamp--Lieb inequalities for \textit{general} input functions $f_j$. For example, gaussian near-minimizers always exist for \eqref{e:IBL} in all non-degenerate situations; this was established in full generality in \cite{BW} (using mass transportation) and in certain special cases in \cite{CDP} (using heat flow in the spirit of the work of Carlen--Lieb--Loss \cite{CLL} and Bennett--Carbery--Christ--Tao \cite{BCCT} on the forward Brascamp--Lieb inequality). Such a result is a counterpart to Lieb's famous theorem for the forward version of \eqref{e:IBL}. In the present work, we investigate what happens when the input functions $f_j$ are ``regularized" in the sense that they coincide with the evolution of positive finite Borel measures under certain heat equations, and our main result is a regularized version of the aforementioned theorem of Barthe and Wolff. 

Our motivation to pursue this project arose from several directions. Firstly, a regularized version of the forward Brascamp--Lieb inequality was considered by Bennett--Carbery--Christ--Tao in \cite{BCCT} and, in particular, they were able to provide a completely new proof of Lieb's theorem based on heat flow (see the forthcoming Theorem \ref{t:BCCT_genLieb} for a version which incorporates a gaussian kernel). This approach to Lieb's theorem was an important source of inspiration to us and we shall see that our analysis of \eqref{e:IBL} for regularized inputs will naturally yield the inverse version of Lieb's theorem in full generality. In this sense, we follow Chen--Dafnis--Paouris \cite{CDP} by adopting heat-flow techniques for the inverse inequality, extend their gaussian saturation result to full generality, and thus re-derive the gaussian saturation result of Barthe--Wolff \cite{BW} by heat-flow regularization. The special case considered in \cite{CDP} is a certain ``geometric" version of the inverse Brascamp--Lieb inequality (see Section \ref{section:prelims} for further details) and the extension to the general case is far from straightforward.

More broadly speaking, the act of regularizing an inequality by restricting the input functions to a smaller subclass of sufficiently well-behaved functions is natural and ubiquitous. One potential benefit of doing so is to raise of possibility of establishing the existence of extremizers for the inequality amongst the restricted subclass, and in doing so may open up a fruitful approach to analyze the general form of the inequality. As we shall see more precisely later, this is the philosophy behind the aforementioned heat-flow proof of Lieb's theorem in \cite{BCCT}, where general input functions are treated as limits of solutions to heat equations at time zero.

Another virtue of restricting the inputs to regularized functions is that one may obtain an improved form of the inequality. Such a perspective is often taken up in fields such as convex geometry, differential geometry, probability and information theory, in which case  the regularization is often described in terms of semi/uniform log-concavity and log-convexity. For example, a famous conjecture of Talagrand predicts that Markov's inequality can be improved if one restricts to regularized functions; see, for example, \cite{EL}. This conjecture was originally stated for the discrete cube and as far as we are aware it is still an open problem in that form, but its analogue for gaussian space has been affirmatively solved by Eldan--Lee \cite{EL}. 

In differential geometry, one can use functional inequalities involving entropy in order to describe the curvature of a Riemannian manifold; this is based on fundamental work of Bakry--Emery and we refer the reader to the book by Bakry--Gentil--Ledoux \cite{BGL} for more details. One example of such an inequality is the (local) log-Sobolev inequality on a weighted Riemannian manifold \cite[Theorem 5.5.2]{BGL}. More recently, several papers have highlighted the importance of improving the log-Sobolev inequality by restricting inputs to regularized functions, which can be regarded as a certain ``curvature improvement". As far as we are aware, the first example of this kind can be found in the work of Fathi--Indrei--Ledoux \cite{FIL} where they improved the best constant for the log-Sobolev inequality under regularization in terms of the Poincar\'{e} constant. Along this line, very recently Aishwarya--Rotem \cite{AiRo} established an improvement of the convexity of entropy under regularization in terms of uniform log-concavity, and related work can also be found in Eldan--Lehec--Shenfeld \cite{ELS} and Bez--Nakamura--Tsuji  \cite{BNT}. Another important entropic inequality in information theory is the Shannon--Stam inequality, which is a special case of the entropic Brascamp--Lieb inequality (see Carlen--Lieb--Loss \cite{CLL} and Carlen--Cordero-Erausquin \cite{CC}). Inspired by important work of Ball--Barthe--Naor \cite{BBN} and Ball--Nguyen \cite{BN} on entropy jump inequalities, the stability problem for the Shannon--Stam inequality has recently attracted attention. In this direction, the regularized framework has played a role; for instance, Courtade--Fathi--Pananjady \cite{CFP} and Eldan--Mikulincer \cite{EM} established stability estimates for the Shannon--Stam inequality for uniformly log-concave random variables.

In convex geometry, regularization appears in terms of the convexity of the boundary of a convex body. 
We mention the work of Schmuckenschl\"{a}ger \cite{Schmu} (see also Klartag--Milman \cite{KlMi}) where regularized convex bodies (2-uniformly convex bodies) were investigated and, in particular, Bourgain's celebrated hyperplane conjecture for such convex bodies was established. Furthermore, the second author, together with Tsuji, pointed out a close link between the Brascamp--Lieb inequality and the volume product of a convex body \cite{NT,NT2}. 
In particular, they developed ideas which arose out of consideration of the regularized Brascamp--Lieb inequality, and succeeded in confirming Mahler's conjecture for certain regularized convex bodies (see \cite{NT}).  

Finally we mention that a regularized version of the forward Brascamp--Lieb inequality of a different (rougher) nature has recently been obtained by Maldague \cite{Maldague} (see also Zorin-Kranich \cite{ZK}), and has found applications to multilinear Kakeya-type inequalities and decoupling theory for the Fourier transform in work of Guo--Oh--Zhang--Zorin-Kranich \cite{GOZZ}. The regularization in \cite{Maldague} has a rough cut-off instead of a gaussian weight factor  and the input functions $f_j$ are constant on cubes in the unit cube lattice.

Before presenting our main results in Section \ref{section:main}, to help set the scene we give an introduction and some background regarding Brascamp--Lieb inequalities. In line with the historical development of the subject, we begin with the forward Brascamp--Lieb inequality. After this, we discuss recent developments regarding the inverse inequality \eqref{e:IBL}. Finally, we introduce the so-called forward-reverse Brascamp--Lieb inequality which originated in Liu--Courtade--Cuff--Verd\'u \cite{LCCV} and whose theory was significantly developed in Courtade--Liu \cite{CouLiu}. The forward-reverse Brascamp--Lieb framework encompasses both the forward and inverse Brascamp--Lieb inequalities as well as the reverse Brascamp--Lieb inequality due to Barthe \cite{Barthe0, Barthe1}. However, it is in fact the case that the gaussian saturation property for the inverse inequality \eqref{e:IBL} implies the gaussian saturation property for the forward-reverse Brascamp--Lieb inequality. This follows from a clever argument of Wolff (e.g. see \cite[Section 4]{CouLiu}), and we shall make use of Wolff's idea to deduce regularized forward-reverse Brascamp--Lieb inequalities from our main result for regularized inverse Brascamp--Lieb inequalities (thus, for example, recovering regularized forward Brascamp--Lieb inequalities due to Bennett--Carbery--Christ--Tao \cite{BCCT}).

\subsection{The forward Brascamp--Lieb inequality} 
Inequalities of the form
\begin{equation} \label{e:FBL}
\int_{\mathbb{R}^n}  \prod_{j=1}^mf_j(B_jx)^{c_j}\, dx \le C  \prod_{j=1}^m \bigg( \int_{\mathbb{R}^{n_j}} f_j \bigg)^{c_j}
\end{equation}
for integrable functions $f_j :\mathbb{R}^{n_j} \to \mathbb{R}_+$ are widely known as \emph{Brascamp--Lieb inequalities}\footnote{We often add the term ``forward" to clarify that we are referring to \eqref{e:FBL} rather than \eqref{e:IBL}.}. Here, the linear transformations $B_j : \mathbb{R}^n \to \mathbb{R}^{n_j}$ and positive exponents $c_j$ are given, and the pair $(\mathbf{B}, \mathbf{c})$ is referred to as a \emph{Brascamp--Lieb datum}, where $\mathbf{B}= (B_j)_{j=1}^m$ and $\mathbf{c} = (c_j)_{j=1}^m$. The best (i.e. smallest) constant $C$ is called the \emph{Brascamp--Lieb constant} and is defined by
\[
{\rm F}(\mathbf{B},\mathbf{c}) = \sup_{\mathbf{f}} {\rm BL}(\mathbf{B},\mathbf{c};\mathbf{f}) \in (0,\infty]
\]
where
\[
{\rm BL}(\mathbf{B},\mathbf{c};\mathbf{f}) = \frac{\int_{\mathbb{R}^n} \prod_{j=1}^mf_j(B_jx)^{c_j}\, dx}{\prod_{j=1}^m ( \int_{\mathbb{R}^{n_j}} f_j )^{c_j}}
\]
is the \emph{Brascamp--Lieb functional}, and the supremum is taken over measurable $f_j :\mathbb{R}^{n_j} \to \mathbb{R}_+$ such that $\int_{\mathbb{R}^{n_j}} f_j \in (0,\infty)$.

The multilinear H\"older, Loomis--Whitney and Young convolution inequalities are standard examples of Brascamp--Lieb inequalities. From a historical perspective, the pursuit of the best constant for the Young convolution inequality was influential in the emergence of the Brascamp--Lieb inequality. Indeed, Beckner \cite{Beck} and Brascamp--Lieb \cite{BL} independently identified that the best constant in non-boundary cases of the Young convolution inequality is attained (essentially uniquely so -- see \cite{BL}) on certain isotropic centred gaussians, and the systematic study of the more general inequality \eqref{e:FBL} traces back to \cite{BL}. With the main focus of the current paper in mind, we also remark that Brascamp and Lieb in \cite{BL} considered the inverse form of the Young convolution inequality in which the direction of the inequality is reversed (and negative exponents $c_j$ are admissible). They were able to obtain the best (i.e. largest) constant, a characterization of maximizers and, as a delightful application, were able to re-derive the Pr\'ekopa--Leindler inequality via a clever limiting argument.

Nowadays the theory of the Brascamp--Lieb inequality is well developed and finds applications across a vastly diverse range of fields. As but one example, we mention again the close  link to multilinear Kakeya-type inequalities and decoupling theory for the Fourier transform, both of which have found staggering applications in the last 15 years to problems in harmonic analysis, geometric analysis, dispersive PDE, and number theory; see, for example \cite{BourgainJAMS, BDG, BourgainGuth, GZ, GZK}. Whilst it feels somewhat discourteous not to include details of other kinds of applications of the Brascamp--Lieb inequality (and its variants), there are a number of papers which already contain thorough discussions of this nature and we encourage the reader to try \cite{BBRIMS, BBBCF, DT, Gardner, Zhang}.

Unlike (non-boundary cases of) the Young convolution inequality, the classes of maximizers for the multilinear H\"older and Loomis--Whitney inequalities are of a wider nature. Nevertheless, it is the case that isotropic centred gaussians are amongst the maximizers for each of these inequalities and this naturally raises the question of whether this is a general phenomenon for \eqref{e:FBL}. Remarkably, Lieb \cite{Lieb} established that centred gaussian \emph{near maximizers} always exist for \eqref{e:FBL}. We state this result next, along with a ``localized" version incorporating a centred gaussian weight factor. For this, we introduce a little more notation.

We extend the notion of Brascamp--Lieb datum to the triple $(\mathbf{B},\mathbf{c},\mathcal{Q})$, where $\mathcal{Q} : \mathbb{R}^n \to \mathbb{R}^n$ is a given self-adjoint transformation. Associated to such a datum is the Brascamp--Lieb constant 
\[
{\rm F}(\mathbf{B},\mathbf{c},\mathcal{Q}) = \sup_{\mathbf{f}} {\rm BL}(\mathbf{B},\mathbf{c},\mathcal{Q};\mathbf{f}) \in (0,\infty]
\]
which is given in terms of the Brascamp--Lieb functional
\[
{\rm BL}(\mathbf{B},\mathbf{c},\mathcal{Q};\mathbf{f}) = \frac{\int_{\mathbb{R}^n} e^{-\pi \langle x, \mathcal{Q}x\rangle} \prod_{j=1}^mf_j(B_jx)^{c_j}\, dx}{\prod_{j=1}^m ( \int_{\mathbb{R}^{n_j}} f_j )^{c_j}}.
\]
In the above, the supremum is taken over measurable $f_j :\mathbb{R}^{n_j} \to \mathbb{R}_+$ such that $\int_{\mathbb{R}^{n_j}} f_j \in (0,\infty)$. Also, associated to a positive definite transformation $A$ on a given euclidean space, we define the (normalized) centred gaussian $g_A$ by
\begin{equation}\label{e:Gaussian}
g_A(x):= ({\rm det}\, A)^{1/2} e^{-\pi \langle x,Ax\rangle}. 
\end{equation} 
On gaussian inputs $\mathbf{f}$ with $f_j = g_{A_j}$ for each $j = 1,\ldots,m$, we slightly abuse notation for the Brascamp--Lieb functional and write 
\[
{\rm BL}(\mathbf{B},\mathbf{c},\mathcal{Q};\mathbf{A}) := {\rm BL}(\mathbf{B},\mathbf{c},\mathcal{Q};\mathbf{f}).
\]

\begin{theorem}[Lieb] \label{t:Lieb}
Let $(\mathbf{B},\mathbf{c})$ be a Brascamp--Lieb datum. Then we have
\[
{\rm F}(\mathbf{B}, \mathbf{c})
=
\sup_{\mathbf{A}} {\rm BL}(\mathbf{B},\mathbf{c};\mathbf{A}),
\]
where the supremum is taken over positive definite transformations $A_j$.
More generally
\begin{equation}\label{e:Lieb}
{\rm F}(\mathbf{B}, \mathbf{c},\mathcal{Q})
=
\sup_{\mathbf{A}} {\rm BL}(\mathbf{B},\mathbf{c},\mathcal{Q};\mathbf{A}) 
\end{equation}
whenever  $\mathcal{Q} : \mathbb{R}^n \to \mathbb{R}^n$ is positive semi-definite.
\end{theorem}
Lieb's theorem above was proved in full generality in \cite{Lieb}. Earlier, Brascamp and Lieb \cite{BL} had established certain special cases including, for instance, rank-one linear transformations $B_j$. A computation reveals that
\begin{equation*} 
{\rm BL}(\mathbf{B}, \mathbf{c}, \mathcal{Q}; \mathbf{A})
= \bigg( \frac{\prod_{j=1}^m ({\rm det}\, A_j)^{c_j}}{ {\rm det}\, (\mathcal{Q}+\sum_{j=1}^m c_j B_j^* A_jB_j)} \bigg)^{1/2}
\end{equation*}
if $\mathcal{Q}+\sum_{j=1}^m c_j B_j^* A_jB_j$ is positive definite (${\rm BL}(\mathbf{B}, \mathbf{c}, \mathcal{Q}; \mathbf{A}) =\infty$ otherwise) and therefore Lieb's theorem dramatically reduces the complexity of understanding the Brascamp--Lieb constant. For example, Lieb's theorem played a pivotal role in the proof in \cite{BBCF} of the continuity of the Brascamp constant $\mathbf{B} \mapsto {\rm F}(\mathbf{B},\mathbf{c})$ and consequently in recent developments in understanding nonlinear variants of the Brascamp--Lieb inequality in which the underlying transformations $B_j$ are allowed to be nonlinear \cite{BBBCF}.

We end this very brief overview of the Brascamp--Lieb inequality \eqref{e:FBL} by stating a theorem from \cite{BCCT, BCCT_MRL} which provides a characterization of the \emph{finiteness} of the Brascamp--Lieb constant (see also recent work of Gressman \cite{Gressman} for new perspectives in this direction). 
\begin{theorem}[Bennett--Carbery--Christ--Tao] \label{t:BCCT_finiteness}
Let $(\mathbf{B},\mathbf{c})$ be a Brascamp--Lieb datum. Then ${\rm F}(\mathbf{B}, \mathbf{c})$ is finite if and only if
\begin{equation} \label{e:BLscaling}
n = \sum_{j=1}^m c_jn_j
\end{equation}
and
\begin{equation} \label{e:BLdimension}
\dim V \leq \sum_{j=1}^m c_j \dim B_jV \quad \text{for all subspaces $V$ of $\mathbb{R}^n$.}
\end{equation}
\end{theorem}
The above result along with Lieb's theorem form two pillars in the theory of the Brascamp--Lieb inequality. The scaling condition \eqref{e:BLscaling} is easily shown to be necessary for the Brascamp--Lieb constant to be finite. Also, as further necessary conditions for finiteness, by considering $V = \mathbb{R}^n$ and $V = \bigcap_{j=1}^m \ker B_j$, one obtains from \eqref{e:BLscaling} and \eqref{e:BLdimension} that each $B_j$ must be surjective and that $\bigcap_{j=1}^m \ker B_j = \{0\}$; such data are usually referred to as \emph{non-degenerate}.

\begin{remark}
Maldague's regularized version of \eqref{e:FBL} quantifies the finiteness characterization in Theorem \ref{t:BCCT_finiteness}. In particular, it is shown in \cite{Maldague} that when the integral on the left-hand side of \eqref{e:FBL} is truncated to a ball of radius $R > 0$ and the input functions are restricted to be constant on cubes in a lattice of size $r \in (0,R)$, then the constant behaves like $R^\kappa r^{-\widetilde{\kappa}}$, where $\kappa = \sup_{V \leq \mathbb{R}^n} (\dim V - \sum_{j=1}^m c_j \dim B_jV)$ and $\widetilde{\kappa} := \kappa - (n - \sum_{j=1}^m c_jn_j)$. On the other hand, the regularized version of \eqref{e:FBL} proved in \cite[Corollary 8.15]{BCCT} may be viewed as an extension of Lieb's theorem, and our main result (Theorem \ref{t:QMain} below) is an extension of the inverse version of Lieb's theorem.
\end{remark}

\subsection{The inverse Brascamp--Lieb inequality} 
For the inverse Brascamp--Lieb inequality \eqref{e:IBL}, it is appropriate to introduce a different notion of non-degeneracy compared with the forward version of the inequality. The non-degeneracy condition was identified by Barthe--Wolff \cite{BW} and is as follows. Assume each $B_j : \mathbb{R}^n \to \mathbb{R}^{n_j}$ is a surjection, and $\mathbf{c}$ is arranged so that 
\[
c_1,\ldots,c_{m_+} >0 > c_{m_++1},\ldots, c_m 
\] 
for some $0\le m_+ \le m$, where we interpret the case $m_+ = 0$ to mean that all of $c_1,\ldots, c_m$ are negative. Correspondingly, we define the linear transformation $\mathbf{B}_+: \mathbb{R}^n \to \oplus_{j=1}^{m_+}\mathbb{R}^{n_j}$ by
$$
\mathbf{B}_+x := (B_1x,\ldots, B_{m_+}x). 
$$
Then, if we denote the number of positive eigenvalues of the self-adjoint transformation $\mathcal{Q} : \mathbb{R}^n \to \mathbb{R}^n$ by $s^+(\mathcal{Q})$, the Brascamp--Lieb datum\footnote{Strictly speaking, this terminology has already been used for the forward Brascamp--Lieb inequality where we imposed the condition $c_j > 0$ for all $j$. It would be more accurate to use terminology such as ``inverse Brascamp--Lieb datum", but it will always be clear from the context which notion of Brascamp--Lieb datum is being used.} $(\mathbf{B},\mathbf{c},\mathcal{Q})$ is said to be \emph{non-degenerate} if 
\begin{equation}\label{e:NondegBW}
\mathcal{Q}|_{{\rm ker}\, \mathbf{B}_+} > 0 \qquad \text{and} \qquad n \ge s^+(\mathcal{Q}) +\sum_{j=1}^{m_+} n_j 
\end{equation}
hold. Here $\mathcal{Q}|_{{\rm ker}\, \mathbf{B}_+}$ is the restriction of $\mathcal{Q}:\mathbb{R}^n\to\mathbb{R}^n$ to the subspace ${\rm ker}\, \mathbf{B}_+$. Note that if $\mathcal{Q} = 0$ then the non-degeneracy condition \eqref{e:NondegBW} is equivalent to fact that $\mathbf{B}_+$ is a bijection; in other words, 
\begin{equation}\label{e:DefNondeg}
\bigcap_{j=1}^{m_+} {\rm ker}\, B_j = \{0\}  \qquad \text{and} \qquad  {\rm Im}\, \mathbf{B}_+ = \bigoplus_{j=1}^{m_+} \mathbb{R}^{n_j}. 
\end{equation}

Next, in analogy with the notation introduced above for the forward Brascamp--Lieb inequality, we define 
\[
{\rm I}(\mathbf{B}, \mathbf{c},\mathcal{Q}) := \inf_{\mathbf{f}} {\rm BL}(\mathbf{B}, \mathbf{c},\mathcal{Q}; \mathbf{f})
\]
to be the best (i.e. largest) constant for which \eqref{e:IBL} holds, where the infimum is taken over measurable $f_j :\mathbb{R}^{n_j} \to \mathbb{R}_+$ such that $\int_{\mathbb{R}^{n_j}} f_j \in (0,\infty)$. Strictly speaking, we are extending our definition of the Brascamp--Lieb functional $\mathbf{f} \mapsto {\rm BL}(\mathbf{B}, \mathbf{c},\mathcal{Q}; \mathbf{f})$ to the case of real exponents $\mathbf{c} \subseteq \mathbb{R}^m$, which we do with the understanding that $0 \cdot \infty = 0$. 

It is not immediately apparent that \eqref{e:NondegBW} is a natural non-degeneracy condition to impose and we refer the reader to the careful discussion in \cite[Section 2]{BW} for the details. Here we simply extract from \cite[Section 2]{BW} that ${\rm I}(\mathbf{B}, \mathbf{c},\mathcal{Q}) = 0$ or $\infty$ when the non-degeneracy condition is dropped. 

The following inverse version of Lieb's theorem was proved in \cite[Theorem 2.9]{BW}.
\begin{theorem}[Barthe--Wolff] \label{t:BWQ}
For any non-degenerate Brascamp--Lieb datum $(\mathbf{B},\mathbf{c},\mathcal{Q})$, we have 
$$
{\rm I}(\mathbf{B},\mathbf{c},\mathcal{Q}) = \inf_{\mathbf{A}} {\rm BL}(\mathbf{B}, \mathbf{c},\mathcal{Q}; \mathbf{A}).
$$
\end{theorem} 
For completeness, we note that a counterpart to Theorem \ref{t:BCCT_finiteness} regarding the strict positivity of $\mathrm{I}(\mathbf{B},\mathbf{c},\mathcal{Q})$ was proved by Barthe--Wolff; see \cite[Theorem 1.5]{BW} for the case $\mathcal{Q} = 0$ and \cite[Theorem 8.9]{BW} for the general case.

\subsection{The forward-reverse Brascamp--Lieb inequality} \label{section:FRBL}
Next we introduce the broad class of inequalities first considered in \cite{LCCV} called \emph{forward-reverse Brascamp--Lieb inequalities}. To do so, we fix linear transformations $T_j : E \to E(j)$, $j = 1,\ldots,J$, with $E(j) = \mathbb{R}^{n(j)}$ and where the base space is given by $E = \oplus_{i=1}^I E_i$ with $E_i =  \mathbb{R}^{n_i}$, $i = 1,\ldots,I$. The collection of such linear transformations is written $\mathbf{T} = (T_j)_{j = 1}^J$, and also we introduce the notation $\pi_i$ for the orthogonal projection from $E$ to $E_i$. Finally, we fix two collections of positive real numbers $(d_i)_{i =1}^I$ and $(d(j))_{j =1}^J$, and write $\mathbf{d} = ((d_i)_{i=1}^I, (d(j))_{j=1}^J )$. 

The forward-reverse Brascamp--Lieb inequality associated with the datum $(\mathbf{T},\mathbf{d})$ takes the form
\begin{equation} \label{e:FRBLnokernel}
\prod_{i=1}^I \bigg( \int_{E_i} f_i \bigg)^{d_i}
		\le C\prod_{j=1}^J \bigg( \int_{E(j)} h_j \bigg)^{d(j)}
\end{equation}
for input functions $f_i : E_i \to \mathbb{R}_+$, $h_j : E(j) \to \mathbb{R}_+$ which satisfy 
\begin{equation} \label{e:FRBLnokernelf}
\prod_{i=1}^I f_i(\pi_i x)^{d_i} \le 
\prod_{j=1}^J h_j(T_jx)^{d(j)} \qquad \text{for all $x \in E$.}
\end{equation}
As described in \cite{CouLiu}, this framework encompasses the forward, reverse and inverse forms of the Brascamp--Lieb inequality in the case where there is no gaussian kernel. For example, taking $I = 1$ and $d_1 = 1$, it is clear that the optimal choice in \eqref{e:FRBLnokernelf} is $f = \prod_{j=1}^J (h_j \circ T_j)^{d(j)}$, in which case \eqref{e:FRBLnokernel} reduces to \eqref{e:FBL}. 

Also, Barthe's reverse form of the Brascamp--Lieb inequality, originating in \cite{Barthe1}, corresponds to the case $J = 1$ and $d(1) = 1$. In other words, given linear mappings $B_i : \mathbb{R}^n \to \mathbb{R}^{n_i}$ and positive real numbers $d_i$, $i = 1,\ldots,I$, the reverse Brascamp--Lieb inequality takes the form
\begin{equation} \label{e:reverseBarthe}
C \prod_{i=1}^I \bigg( \int_{\mathbb{R}^{n_i}} f_i \bigg)^{d_i} \leq \int_{\mathbb{R}^n} h_1
\end{equation}
for functions satisfying
\begin{equation}\label{e:CondRBarthe}
\prod_{i=1}^I f_i(x_i)^{d_i} \leq h_1\bigg( \sum_{i=1}^I d_i B_i^*x_i \bigg)  \qquad \text{for $(x_1,\ldots,x_I) \in \oplus_{i = 1}^I \mathbb{R}^{n_i}$.}
\end{equation}
As has been pointed out elsewhere, arguably it would be more natural to refer to inequalities of the form \eqref{e:reverseBarthe} as \emph{dual} (rather than reverse) Brascamp--Lieb inequalities in light of the fact that  
\begin{equation} \label{e:BLduality}
\mathrm{R}(\mathbf{B},\mathbf{d}) \mathrm{F}(\mathbf{B},\mathbf{d})  = 1,
\end{equation}
where $\mathrm{R}(\mathbf{B},\mathbf{d})$ denotes the best (i.e. largest) constant $C$ such that \eqref{e:reverseBarthe} holds under \eqref{e:CondRBarthe}. This remarkable fact was proved by Barthe in \cite{Barthe1}. 

We also remark that the Pr\'ekopa--Leindler inequality is the special case of \eqref{e:reverseBarthe} with $C=1$ obtained by taking $I=2$, $B_1 = B_2 = \mathrm{id}_{\mathbb{R}^n}$, and $d_1 + d_2 = 1$. From the Pr\'ekopa--Leindler inequality one can derive the famous Brunn--Minkowski inequality
\[
\mathrm{vol}_n(X + Y)^{1/n} \geq \mathrm{vol}_n(X)^{1/n} + \mathrm{vol}_n(Y)^{1/n}
\] 
for appropriate $X,Y \subseteq \mathbb{R}^n$. The stimulus to introduce the general form of the inequality \eqref{e:reverseBarthe} appears to have arisen from convex geometry and we refer the reader to \cite{BallJLMS, Barthe1, Barthesimplex} for further discussion and applications.

Finally, to deduce the inverse Brascamp--Lieb inequality, take a forward-reverse Brascamp--Lieb datum $(\mathbf{T},\mathbf{d})$ and set
\[
h_{J + 1}(x) = \prod_{i=1}^I f_i(\pi_i x)^{d_i}  \prod_{j=1}^J h_j(T_jx)^{-d(j)}.
\]
Then \eqref{e:FRBLnokernelf} holds if we replace $J$ by $J + 1$, take $T_{J + 1}$ to be the identity transformation, and $d_{J + 1} = 1$. The resulting inequality \eqref{e:FRBLnokernel} reduces to the inverse Brascamp--Lieb inequality
\begin{equation*} 
\prod_{i=1}^I \bigg( \int_{E^i} f_i \bigg)^{d_i} \prod_{j=1}^J \bigg( \int_{E(j)} h_j \bigg)^{-d(j)} \le C \int_E \prod_{i=1}^I f_i(\pi_i x)^{d_i}  \prod_{j=1}^J h_j(T_jx)^{-d(j)} \, dx.
\end{equation*}
With the non-degeneracy condition in mind (i.e. $\mathbf{B}_+$ is bijective), this framework is as general as the one presented in the previous section in the case of no gaussian kernel.\footnote{We also refer the reader to the Appendix for a more complete discussion along these lines.}

The analogue of Lieb's theorem holds in the setting of the forward-reverse Brascamp--Lieb inequality, as shown in \cite[Theorem 2]{LCCV}.
\begin{theorem}[Liu--Courtade--Cuff--Verd\'u] \label{t:LCCV}
For any datum $(\mathbf{T},\mathbf{d})$, if the input functions $f_i : E_i \to \mathbb{R}_+$, $h_j : E(j) \to \mathbb{R}_+$ satisfy \eqref{e:FRBLnokernelf}, then they also satisfy \eqref{e:FRBLnokernel} with constant $C$ given by
\[
\sup_{\substack{A_i > 0 \\ A(j) > 0}} \prod_{i=1}^I ({\rm det}\, A_i)^{-\frac{d_i}2} \prod_{j=1}^J ( {\rm det}\, A(j) )^{\frac{d(j)}2}. 
\]
\end{theorem}
The proof of Theorem \ref{t:LCCV} in \cite{LCCV} rests on an equivalent entropic formulation of the forward-reverse Brascamp--Lieb inequality\footnote{In the context of the forward Brascamp--Lieb inequality, duality with subadditivity of entropy can be found in \cite{CLL} and \cite{CC}.} and ideas similar to those in Geng--Nair \cite{GN} and Lieb's original proof of the gaussian saturation property for the forward Brascamp--Lieb inequality in \cite{Lieb}. A different proof of Theorem \ref{t:LCCV} was found by Courtade--Liu \cite{CouLiu}, again utilizing entropic duality, but incorporating ideas from Bennett--Carbery--Christ--Tao \cite{BCCT} and Lehec \cite{Lehec}. We also remark that \cite{CouLiu} embarked on a systematic study of the forward-reverse Brascamp--Lieb inequality and, for example, established a characterization of the finiteness of the constant in the spirit of Theorem \ref{t:BCCT_finiteness} (see \cite[Theorem 1.27]{CouLiu}) and extended Barthe's duality relation \eqref{e:BLduality} to all forward-reverse Brascamp--Lieb data (see \cite[Theorem 1.3]{CouLiu}).

\begin{remark}
Given the above discussion, it is clear that one may use Theorem \ref{t:LCCV} to deduce Theorem \ref{t:BWQ} in the case where there is no gaussian kernel. Surprisingly, the converse is true and this fact follows from a clever limiting argument due to Pawe\l{} Wolff which is explained in \cite[Section 4.1]{CouLiu}. In the presence of a gaussian kernel, as far as we are aware, a version of such an equivalence has not appeared in the literature (see \cite[Section 4]{CouLiu} for some partial results along these lines).
\end{remark}

In the following section we state our main result -- an extension of Theorem \ref{t:BWQ} to certain regularized input functions -- and some consequences. For example, we can recover all versions of the gaussian saturation property stated above (i.e. Theorems \ref{t:Lieb}, \ref{t:BWQ} and \ref{t:LCCV}). In fact, we present a framework\footnote{This framework was suggested to us by an anonymous referee to whom we are extremely grateful.} for the forward-reverse Brascamp--Lieb inequality which allows for gaussian kernels, and, in particular, we shall see that the gaussian saturation property in this framework is equivalent to that of Barthe--Wolff in Theorem  \ref{t:BWQ}. Despite this equivalence between forward-reverse and inverse Brascamp--Lieb inequalities, our heat-flow monotonicity approach seems best suited to the inverse Brascamp--Lieb inequality and so the inverse inequality should be regarded as the focal point of the paper. In fact, it is not at all clear to us whether one can expect a heat-flow monotonicity approach to succeed in the setting of the general forward-reverse Brascamp--Lieb inequality. Heat-flow arguments of a different nature have been successfully implemented in the case of the reverse Brascamp--Lieb inequality with geometric data -- see Barthe--Cordero-Erausquin \cite{BC-E} (rank one transformations) and Barthe--Huet \cite{BH} (general rank transformations), as well as papers by Borell \cite{Borell1, Borell2, Borell3} which seem to have been a source of inspiration for \cite{BC-E} and \cite{BH}. For example, it is shown in \cite[Theorem 4]{BH} that, for geometric data, the relation \eqref{e:CondRBarthe} is preserved if one replaces all functions with their evolution under classical heat flow $e^{t\Delta}$ \emph{for all} $t > 0$; the reverse Brascamp--Lieb inequality for geometric data then follows by taking a limit $t \to \infty$. It would be very interesting to see if such an approach can be extended to general data in the forward-reverse framework, but it is certainly not clear to us whether such an approach can be profitable in proving, say, Theorem \ref{t:RegFRBL}.

\section{Main results} \label{section:main}

\subsection{Regularized inverse Brascamp--Lieb inequalities}
From now on, we denote the class of $n\times n$ real and self-adjoint transformations by $S(\mathbb{R}^n)$ and 
\begin{equation*}
S_+(\mathbb{R}^n) := \{ A \in S(\mathbb{R}^n): A>0 \}
\end{equation*}
for the subclass of positive definite transformations. 

For $G\in S_+(\mathbb{R}^n)$ consider the solution $u : \mathbb{R}_+ \times \mathbb{R}^n \to \mathbb{R}_+$ to the heat equation 
\begin{equation}\label{e:HeatG}
\partial_t u = \frac1{4\pi} {\rm div}\, (G^{-1} \nabla u), \qquad  u(0) = \mu,
\end{equation}
where $\mu$ is a positive finite Borel measure with non-zero mass. Then we call $f(x)= u(1,x)$ as a \textit{type $G$ function}. Explicitly $u(1,x)$ can be written in convolution form
\[
u(1,x) = g_{G}\ast \mu(x) = (\det G)^{1/2} \int_{\mathbb{R}^n} e^{-\pi \langle x-y,G(x-y) \rangle} \, d\mu(y).
\]
We consider the inverse Brascamp--Lieb inequality restricted to regularized input functions of this type, and thus we introduce the notation
\begin{equation*}\label{e:TypeG}
\mathcal{T}(G):= \{ g_{G}\ast \mu : \text{$\mu$ is a positive finite Borel measure with non-zero mass} \}. 
\end{equation*}
Note that the class $\mathcal{T}(G)$ is clearly subset of the class of all non-negative and integrable functions, and its members are smooth and strictly positive. Also, it is not difficult to see that we have the nesting property
\begin{equation} \label{e:nesting}
G_1,G_2\in S_+(\mathbb{R}^n), \,\,  G_1\le G_2 \,\, \Rightarrow \,\, \mathcal{T}(G_1) \subseteq \mathcal{T}(G_2). 
\end{equation}
It also formally makes sense\footnote{For example, we know that $g_{\lambda {\rm id}}$ converges to the Dirac delta supported at the origin as $\lambda \to \infty$.} to view
\begin{equation}\label{Def:Typeinfty}
\mathcal{T}(\infty) = \bigg\{ f:\mathbb{R}^n\to \mathbb{R}_+ : \; \int_{\mathbb{R}^n} f < \infty \bigg\}
\end{equation}   
as the class of all inputs. 

For each Brascamp--Lieb datum $(\mathbf{B},\mathbf{c},\mathcal{Q})$ and $\mathbf{G} = (G_j)_{j=1}^m$, with $G_j \in S_+(\mathbb{R}^{n_j})$, we refer to $(\mathbf{B},\mathbf{c},\mathcal{Q},\mathbf{G})$ as a \textit{generalized Brascamp--Lieb datum}\footnote{When $\mathcal{Q} = 0$, we shall simply say that $(\mathbf{B},\mathbf{c},\mathbf{G})$ is a \textit{generalized Brascamp--Lieb datum}.} and define 
\begin{equation}\label{Def:GIBL}
{\rm I}(\mathbf{B}, \mathbf{c}, \mathcal{Q}, \mathbf{G})
:= 
\inf_{\mathbf{f} \in \mathcal{T}(\mathbf{G})} {\rm BL}(\mathbf{B},\mathbf{c},\mathcal{Q};\mathbf{f})
\end{equation}
to be the best (i.e. largest) constant such that \eqref{e:IBL} holds for input functions $f_j \in \mathcal{T}(G_j)$; accordingly, the notation $\mathbf{f} \in \mathcal{T}(\mathbf{G})$ means $f_j \in \mathcal{T}(G_j)$ for each $j=1,\ldots,m$.

For generalized Brascamp--Lieb data $(\mathbf{B},\mathbf{c},\mathcal{Q},\mathbf{G})$, we add a further condition to the non-degeneracy condition \eqref{e:NondegBW}. We shall say that  $(\mathbf{B},\mathbf{c},\mathcal{Q},\mathbf{G})$ is non-degenerate if \eqref{e:NondegBW} holds and
\begin{equation}\label{e:NondegG}
\mathcal{Q} + \sum_{j=1}^{m_+} c_j B_j^* G_j B_j >0.
\end{equation}
This appears to be a reasonable condition in the following sense. If $(\mathbf{B},\mathbf{c},\mathcal{Q})$ is non-degenerate (that is, \eqref{e:NondegBW} holds), then as discussed in the proof of \cite[Proposition 2.2]{BW}, there exists $A_1,\ldots,A_{m_+} > 0$ such that 
\begin{equation*}
\mathcal{Q} + \sum_{j=1}^{m_+} c_j B_j^* A_j B_j >0
\end{equation*}
and hence $(\mathbf{B},\mathbf{c},\mathcal{Q},\mathbf{G})$ is non-degenerate whenever $G_j \geq A_j$ for $j=1,\ldots,m_+$. From such considerations, our main result below recovers the Barthe--Wolff result in Theorem \ref{t:BWQ} as a limiting case; see the remark after Lemma \ref{l:limit} for further details. In addition, we note that $\inf_{\mathbf{A} \leq \mathbf{G}} \mathrm{BL}(\mathbf{B},\mathbf{c},\mathcal{Q};\mathbf{A}) = \infty$ if \eqref{e:NondegG} is not satisfied.
\begin{theorem}\label{t:QMain}
For any non-degenerate generalized Brascamp--Lieb datum $(\mathbf{B}, \mathbf{c},\mathcal{Q},\mathbf{G})$, we have
\begin{equation}\label{e:RevLiebG}
{\rm I} (\mathbf{B}, \mathbf{c},\mathcal{Q},\mathbf{G}) = \inf_{\mathbf{A} \leq \mathbf{G}} {\rm BL}(\mathbf{B},\mathbf{c},\mathcal{Q};\mathbf{A}).
\end{equation}
\end{theorem}
As we have already mentioned, we use a heat-flow approach to prove Theorem \ref{t:QMain}. This may be viewed as a significant extension of the analysis in \cite{CDP} which handled a special class of Brascamp--Lieb data (so-called geometric data -- see the discussion at the beginning of Section \ref{section:prelims}).

\subsection{Regularized forward-reverse Brascamp Lieb inequalities} \label{subsection:RegFRBL}

We shall see that Theorem \ref{t:QMain} yields a gaussian saturation property for a regularized version of the forward-reverse Brascamp--Lieb inequality with gaussian kernels. To present this application, as far as possible, we use similar notation to that in Section \ref{section:FRBL}. 

We fix a collection of linear transformations $\mathbf{T} = (T_j)_{j = 1}^J$. Here, $T_j : E \to E(j)$, $j = 1,\ldots,J$, with $E(j) = \mathbb{R}^{n(j)}$ and where the base space is given by $E = \oplus_{i=0}^I E_i$ with $E_i =  \mathbb{R}^{n_i}$, $i = 0,1,\ldots,I$. Next, we take two collections of positive real number $\mathbf{d} = ((d_i)_{i=1}^I, (d(j))_{j=1}^J )$. In addition, fix $Q_L\in S_+(E_0)$, $Q_R$ to be a positive semi-definite transformation on $E$, and write $Q= (Q_L,Q_R)$. Finally, we fix two collections of positive definite transformations $\mathbf{G}=( (G_i)_{i=1}^I, (G(j))_{j=1}^J )$, where $G_i \in S_+(E_i)$ and $G(j) \in S_+(E(j))$.

The datum $(\mathbf{T}, \mathbf{d}, Q,\mathbf{G} )$ is said to be non-degenerate if 
\begin{equation}\label{e:NondegFRBL-Nov12}
	E_0 \oplus \{0\} \oplus \cdots \oplus \{0\} \subseteq {\rm ker}\, Q_R,\quad \pi_0^* Q_L \pi_0 + \sum_{i=1}^I d_i \pi_i^* G_i \pi_i > Q_R \quad \textrm{on $E$,}
\end{equation}
where, as before, $\pi_i$ denotes the orthogonal projection from $E$ to $E_i$. The second condition is reasonable since it is necessary for \eqref{e:AssumpFRBL} below to hold with\footnote{In this setting it is slightly more convenient to  work with non-normalized gaussians.} $f_i = \gamma_{G_i}, h_j = \gamma_{G(j)}$, where $\gamma_A(x):= e^{-\pi\langle x, Ax\rangle}$. In fact, if the second condition in \eqref{e:NondegFRBL-Nov12} fails to be true, then there is no $A_i \le G_i, A(j) \le G(j)$ satisfying \eqref{e:AssumpFRBL} with $f_i = \gamma_{A_i}, h_j = \gamma_{A(j)}$. 
Indeed, \eqref{e:AssumpFRBL} with $f_i = \gamma_{A_i}, h_j = \gamma_{A(j)}$ is equivalent to 
\begin{equation}\label{e:GaussAssump}
	\pi_0^* Q_L \pi_0 
	+ \sum_{i=1}^I d_i \pi_i^* A_i \pi_i
	\ge 
	Q_R 
	+
	\sum_{j=1}^J d(j) T_j^* A(j) T_j
	\quad \textrm{on $E$,}
\end{equation}
which in particular implies 
$$
	\pi_0^* Q_L \pi_0 
	+ \sum_{i=1}^I d_i \pi_i^* A_i \pi_i
	>
	Q_R.
$$
Under such a non-degeneracy condition, we have the following generalization of Theorem \ref{t:LCCV}.
\begin{theorem}\label{t:RegFRBL}
	For any non-degenerate data $\mathfrak{D} = (\mathbf{T}, \mathbf{d}, Q,\mathbf{G} )$, we have the following. If $f_i \in \mathcal{T}(G_i)$ and $h_j \in \mathcal{T}(G(j))$ satisfy 
	\begin{equation}\label{e:AssumpFRBL}
		e^{ -\pi \langle \pi_0 x,Q_L \pi_0 x \rangle }
		\prod_{i=1}^I f_i(\pi_i x)^{d_i}
		\le 
		e^{ - \pi \langle x,Q_Rx\rangle } 
		\prod_{j=1}^J h_j(T_jx)^{d(j)} \qquad \text{for all $x \in E$,}
	\end{equation}
 then they also satisfy 
	\begin{equation}\label{e:ConsFRBL}
		\prod_{i=1}^I \bigg( \int_{E_i} f_i \bigg)^{d_i}
		\le {\rm FR}(\mathfrak{D}) \prod_{j=1}^J \bigg( \int_{E(j)} h_j \bigg)^{d(j)},
	\end{equation}
	where the constant is given by 
	$$
	{\rm FR}(\mathfrak{D}) =\sup\bigg\{ \prod_{i=1}^I ({\rm det}\, A_i)^{-\frac{d_i}2} \prod_{j=1}^J ( {\rm det}\, A(j) )^{\frac{d(j)}2} : A_i\le G_i, A(j)\le G(j)\;{\rm s.t.}\; \eqref{e:GaussAssump}\bigg\}. 
	$$
\end{theorem}
We shall deduce the above theorem from Theorem \ref{t:QMain} by means of Wolff's limiting argument alluded to at the end of Section \ref{section:intro}. In fact, the theorems are equivalent and we shall prove this in the Appendix. 

Capitalizing on the fact that the forward Brascamp--Lieb inequality is a special case of the forward-reverse Brascamp--Lieb inequality, we can quickly deduce the following generalized version of Lieb's theorem. 
\begin{theorem} \label{t:BCCT_genLieb}
For any generalized Brascamp--Lieb data $(\mathbf{B},\mathbf{c},\mathcal{Q},\mathbf{G})$ with $c_j > 0$ for each $j$, and positive semi-definite $\mathcal{Q}$, we have 
$$
\sup_{\mathbf{f} \in \mathcal{T}(\mathbf{G})} {\rm BL}(\mathbf{B},\mathbf{c}, \mathcal{Q};\mathbf{f})
= 
\sup_{\mathbf{A} \leq \mathbf{G}} {\rm BL}(\mathbf{B},\mathbf{c},\mathcal{Q};\mathbf{A}). 
$$
\end{theorem}
A limiting argument shows that Theorem \ref{t:BCCT_genLieb} implies Theorem \ref{t:Lieb}, so in this sense, we see that Lieb's result is also a consequence of Theorem \ref{t:QMain}. Also, the case $\mathcal{Q} = 0$ was proved by Bennett--Carbery--Christ--Tao \cite[Corollary 8.11]{BCCT} and we shall follow the approach taken in \cite{BCCT} in proving Theorem \ref{t:QMain}.

In a similar manner, one may also quickly obtain a regularized version of Barthe's reverse Brascamp--Lieb inequality \eqref{e:reverseBarthe} from Theorem \ref{t:RegFRBL} (by taking $J = 1$ and $d(1) = 1$).

\begin{remark}
Valdimarsson \cite{V} obtained a regularized version of \eqref{e:reverseBarthe} in which the input functions took the form $f_j(x) = \exp(-\pi \langle x,G_j^{-1}x \rangle - H_j(x))$ with $H_j$ a convex function (so-called ``inverse class $G_j$"). This particular set-up appears to be independent from ours in Theorem \ref{t:RegFRBL}. We  remark that Valdimarsson's setting of inverse class $G_j$ input functions allowed him to extend the ideas of Barthe \cite{Barthe1} and obtain a generalization of the duality relation \eqref{e:BLduality} involving $\mathrm{F}(\mathbf{B},\mathbf{c},\mathbf{G})$. It is unclear to us whether a duality result is possible for the regularization we consider in Theorem \ref{t:RegFRBL} and it seems it may be natural to adapt the framework somehow to include inverse class $G$ functions in some appropriate way.

In addition, a certain regularized version of the forward-reverse Brascamp--Lieb inequality (in its dual entropic representation) was considered by Liu--Courtade--Cuff--Verd\'u in  \cite[Section 4]{LCCV}. Again, we believe that this has no direct connection with the kind of regularization that we study in the present paper.
\end{remark}

\begin{organization}
The remainder of the paper is primarily devoted to proving Theorem \ref{t:QMain}. Section \ref{section:prelims} contains several preliminary observations, mostly related to the heat-flow monotonicity approach that we will use to prove Theorem \ref{t:QMain}. Although the main body of the proof of Theorem \ref{t:QMain} is given in Section \ref{section:MainProof}, the key heat-flow result needed for the proof has been isolated in Theorem \ref{t:cp_IBL} in Section \ref{section:cp}; we take the opportunity to present this result in the form of a ``closure property" for sub/supersolutions to certain heat equations and thus contribute to the emerging theory of such closure properties in, for example, \cite{ABBMMS, BB, BBCrell}. After the proof of Theorem \ref{t:QMain} in Section \ref{section:MainProof}, in Section \ref{section:furtherapps} we present some further applications and remarks. Finally, in the Appendix, we prove the equivalence between Theorems \ref{t:QMain} and \ref{t:RegFRBL}.
\end{organization}

\section{Preliminaries} \label{section:prelims}

It will be convenient to introduce the notation
\[
{\rm I}_{\mathrm{g}}(\mathbf{B},\mathbf{c},\mathcal{Q},\mathbf{G}) = \inf_{\mathbf{A} \leq \mathbf{G}} \mathrm{BL}(\mathbf{B},\mathbf{c},\mathcal{Q};\mathbf{A}) 
\]
for the best constant such that \eqref{e:IBL} holds for gaussian input functions $f_j = g_{A_j}$ with $A_j \leq G_j$ for each $j$. When $\mathbf{G} = (\infty,\ldots,\infty)$ or $\mathcal{Q} = 0$ we simply omit from the above notation.

\subsection{Geometric data and a key linear algebraic lemma}

Before embarking on the full proof of Theorem \ref{t:QMain}, as a highly instructive preliminary first step, let us consider the so-called \emph{geometric} case and, for simplicity, we set $\mathcal{Q} = 0$ and $\mathbf{G} = (\infty,\ldots,\infty)$. The geometric condition on the data  is
\begin{equation} \label{e:GeoCondRev}
 B_jB_j^* = \mathrm{id} \quad (j=1,\ldots,m), \quad \text{and} \quad \sum_{j=1}^m c_j B_j^*B_j = \mathrm{id}.
\end{equation}
In a such a case, it is clear that $\ker \mathbf{B}_+ = \{0\}$ and it follows that the non-degeneracy condition \eqref{e:NondegBW} is equivalent to the surjectivity of $\mathbf{B}_+$.
\begin{theorem} \label{t:Geo}
Suppose the Brascamp--Lieb datum $ (\mathbf{B}, \mathbf{c})$ satisfies \eqref{e:GeoCondRev} and non-degenerate in the sense that $\mathbf{B}_+$ is surjective.
Then   
$$
{\rm I} (\mathbf{B},\mathbf{c}) = {\rm I}_{\mathrm{g}} (\mathbf{B},\mathbf{c}) = {\rm BL} (\mathbf{B},\mathbf{c};\mathbf{A}) = 1, 
$$
where $A_j$ is the identity transformation on $\mathbb{R}^{n_j}$ for $j=1,\ldots,m$.
\end{theorem}
Theorem \ref{t:Geo} was proved in Barthe--Wolff \cite[Theorem 4.7]{BW} based on a mass transportation argument.  A very closely related result was established using heat flow by Chen--Dafnis--Paouris in \cite[Theorem 2]{CDP}; in particular, it was shown that 
\begin{equation} \label{e:GeoRev}
\int_{\mathbb{R}^n} \prod_{j=1}^m f_j (B_jx)^{c_j}\, dx \ge \prod_{j=1}^m \bigg( \int_{\mathbb{R}^{n_j}} f_j \bigg)^{c_j}
\end{equation}
holds under the assumption that $B_jB_j^* = \mathrm{id}$ for each $j$, $n = \sum_{j=1}^m c_j n_j$ and 
\begin{equation} \label{e:CDP}
\mathbf{B} \mathbf{B}^* \geq C^{-1}
\end{equation}
hold, where $\mathbf{B} = (B_1,\ldots,B_m) : \mathbb{R}^n \to \prod_{j=1}^m \mathbb{R}^{n_j}$, and 
\[
C := \mathrm{diag}\, (c_1\mathrm{id}_{\mathbb{R}^{n_1}},\ldots,c_m\mathrm{id}_{\mathbb{R}^{n_m}}).
\]
We refer the reader to \cite[Section 4]{BW} for a precise clarification of how one may obtain \cite[Theorem 2]{CDP} from their work.

As observed in \cite{CDP}, assumption \eqref{e:CDP} is clearly equivalent to
\begin{equation}\label{e:GeoKey}
\bigg| \sum_{j=1}^m c_j B_j^*\mathbf{v}_j \bigg|^2  \ge  \sum_{j=1}^m c_j |\mathbf{v}_j|^2  \qquad \text{for all $\mathbf{v}_j \in \mathbb{R}^{n_j}$,}
\end{equation}
and this fact was key to their heat-flow proof of \eqref{e:GeoRev}. Here we exhibit a sketch proof of Theorem \ref{t:Geo}, following the heat flow argument used in \cite{CDP} (which itself is based on similar heat flow proofs of the forward Brascamp--Lieb inequality in \cite{BCCT}, \cite{CLL}, \cite{Vald_heat}), which will help to understand our argument in the case of more general non-degenerate Brascamp--Lieb data $(\mathbf{B},\mathbf{c},\mathbf{G},\mathcal{Q})$ in Section \ref{section:MainProof}. The geometric case is particularly well suited to a heat-flow argument since gaussian maximizers always exist in such a case (in fact, as stated above, isotropic gaussians are maximizers). Including some details in the specific case of geometric data will also motivate the crucial linear algebraic result (a generalization of \eqref{e:GeoKey} in the forthcoming Lemma \ref{l:Key2}) which underpins the heat-flow argument in our proof of Theorem \ref{t:QMain}.

\emph{A heat-flow approach to Theorem \ref{t:Geo}}.
From \eqref{e:GeoCondRev}, it is easy to check that  
$$
{\rm BL} (\mathbf{B},\mathbf{c};\mathbf{A}) = 1
$$ if each $A_j$ is the identity transformation, so it suffices to prove
\eqref{e:GeoRev} holds for sufficiently nice non-negative $f_j$. (We remark that identifying a sufficiently nice class of test functions will be taken up in Section \ref{subsection:nicef}.) To this end, we define the quantity 
\begin{equation} \label{e:GeoQ}
\mathfrak{Q}(t):= \int_{\mathbb{R}^n} U(t,x) \, dx, 
\end{equation}
where $u_j : (0,\infty) \times \mathbb{R}^{n_j} \to (0,\infty)$ solves the heat equation
\begin{equation} \label{e:isoheat}
\partial_t u_j = \Delta u_j, \qquad u_j(0) = f_j
\end{equation}
and $U := \prod_{j=1}^m (u_j \circ B_j)^{c_j}$. Formally, we have 
\begin{align*}
\mathfrak{Q}(0) & = \int_{\mathbb{R}^n} \prod_{j=1}^m f_j (B_jx)^{c_j}\, dx \\
\lim_{t \to \infty} \mathfrak{Q}(t) & = \prod_{j=1}^m \bigg( \int_{\mathbb{R}^{n_j}} f_j \bigg)^{c_j},
\end{align*}
where the argument for the limit at infinity made use of \eqref{e:GeoCondRev}. Thus our aim is to show that $\mathfrak{Q}$ is non-increasing in time. To this end, we note that
\begin{align*}
\partial_t U(t,x) & = U(t,x) \sum_{j=1}^m c_j \partial_t \log u_j (t,B_jx) \\
& = U(t,x) \sum_{j=1}^m c_j[|\mathbf{v}_j|^2 + {\rm div} (\mathbf{v}_j)](t,B_jx),
\end{align*}
where $\mathbf{v}_j := \nabla \log u_j$. By the divergence theorem, it follows (at least formally) that
\begin{align*}\label{e:Q'Geo}
\mathfrak{Q}'(t) = \int_{\mathbb{R}^n} U(t,x) \bigg[  \sum_{j=1}^m c_j |\mathbf{v}_j(t,B_jx)|^2 - \bigg| \sum_{j=1}^m c_j B_j^*\mathbf{v}_j(t,B_jx) \bigg|^2\bigg]\, dx.
\end{align*}
Thus, the argument above has reduced the proof of Theorem \ref{t:Geo} to showing \eqref{e:GeoKey}. This fact is a special case of Lemma \ref{l:Key2} below (with $A_j = {\rm id}$ for each  $j=1,\ldots,m$); the general case in Lemma \ref{l:Key2} will be crucial for the argument in Section \ref{section:MainProof} to prove Theorem \ref{t:QMain}. Before stating the lemma, we remark that in the case of the forward Brascamp--Lieb inequality with $c_j > 0$ for each $j$, for geometric data $(\mathbf{B},\mathbf{c})$ we have that \eqref{e:GeoKey} holds \emph{in reverse}; to see this, if we define 
\[
\overline{\mathbf{w}} := \sum_{j=1}^m c_jB_j^*\mathbf{v}_j,
\]
then by the Cauchy--Schwarz inequality and the geometric condition $\sum_{j=1}^m c_j B_j^*B_j = \mathrm{id}$, we get
\begin{align*}
|\overline{\mathbf{w}} |^2 = \sum_{j=1}^m \langle \sqrt{c_j}B_j\overline{\mathbf{w}},\sqrt{c_j}\mathbf{v}_j \rangle & \leq \bigg(\sum_{j=1}^m c_j |B_j\overline{\mathbf{w}}|^2 \bigg)^{1/2} \bigg(\sum_{j=1}^m c_j |\mathbf{v}_j|^2\bigg)^{1/2} \\
& = |\overline{\mathbf{w}} | \bigg(\sum_{j=1}^m c_j |\mathbf{v}_j|^2\bigg)^{1/2}
\end{align*}
which rearranges to give \eqref{e:GeoKey} in reverse. Interestingly, it was shown in \cite{BW} that one can obtain \eqref{e:GeoKey} from $\sum_{j=1}^m c_j |\mathbf{v}_j|^2 \le  | \sum_{j=1}^m c_j B_j^*\mathbf{v}_j |^2 $ ($c_j > 0$); more generally, a proof along such lines was used to derive the more general inequality in \eqref{e:GeneKey} below.
\begin{lemma}\label{l:Key2}
Suppose the Brascamp--Lieb datum $(\mathbf{B},\mathbf{c})$ is such that $\mathbf{B}_+$ is surjective, and $A_j \in S_+(\mathbb{R}^{n_j})$, $j = 1,\ldots,m$, are such that the transformation $M := \sum_{j=1}^m c_jB_j^* A_j B_j$ is positive definite.  Let $\mathbf{v}_j \in \mathbb{R}^{n_j}$ for $j=1,\ldots, m$, and let $x_* \in \mathbb{R}^n$ be any non-zero element of $\mathbf{B}_+^{-1}(A_1^{-1}\mathbf{v}_1,\ldots, A_{m_+}^{-1}\mathbf{v}_{m_+})$. Then we have the identity
\begin{equation}\label{e:GeneId}
\langle \overline{\mathbf{w}} , M^{-1} \overline{\mathbf{w}} \rangle -\sum_{j=1}^m c_j \langle \mathbf{v}_j, A_j^{-1} \mathbf{v}_j \rangle
= \langle \overline{\mathbf{w}}', M\overline{\mathbf{w}}'\rangle + 
\sum_{j=m_++1}^m |c_j| \langle \mathbf{v}_j', A_j \mathbf{v}_j' \rangle,
\end{equation}
where  $\overline{\mathbf{w}} := \sum_{j=1}^m c_j B_j^* \mathbf{v}_j$, $\overline{\mathbf{w}}' := x_* - M^{-1} \overline{\mathbf{w}}$, and $\mathbf{v}_j' := B_jx_* - A_j^{-1} \mathbf{v}_j$ for $j = m_++1,\ldots,m$. In particular, we have 
\begin{equation}\label{e:GeneKey}
\sum_{j=1}^m c_j \langle \mathbf{v}_j, A_j^{-1} \mathbf{v}_j \rangle \le \langle \overline{\mathbf{w}} , M^{-1} \overline{\mathbf{w}} \rangle 
\end{equation}
for any $\mathbf{v}_j \in \mathbb{R}^{n_j}$, $j=1,\ldots, m$.
\end{lemma}
Inequality \eqref{e:GeneKey} can be obtained directly from \cite[Lemma 3.5]{BW}. Since the above result is pivotal to the proof of Theorem \ref{t:QMain}, we include our own brief justification. 
\begin{proof}[Proof of Lemma \ref{l:Key2}]
First note that 
\begin{equation}\label{e:GenChoice}
B_jx_* = A_j^{-1}\mathbf{v}_j \quad (1\le j \le m_+)
\end{equation}
and with this in mind we will expand the right-hand side of \eqref{e:GeneId}. 
For $j = m_++1,\ldots,m$,  we have
\begin{align*}
\langle \mathbf{v}_j', A_j \mathbf{v}_j' \rangle 
&= 
\langle B_j x_*, A_j B_jx_*\rangle 
-2 
\langle B_j x_*, \mathbf{v}_j\rangle  
+ 
\langle A_j^{-1}\mathbf{v}_j, \mathbf{v}_j\rangle
\end{align*}
and 
\begin{align*}
\langle \overline{\mathbf{w}}', M\overline{\mathbf{w}}' \rangle 
&= 
\langle x_* , Mx_*  \rangle 
-
2\langle x_* ,  \overline{\mathbf{w}}  \rangle 
+
\langle M^{-1}\overline{\mathbf{w}} ,  \overline{\mathbf{w}}  \rangle. 
\end{align*}
So, if we set $M_+:= \sum_{j=1}^{m_+} c_j B_j^* A_j B_j$, the right-hand side of \eqref{e:GeneId} can be written $I_a+I_b+I_c$ where 
\begin{align*}
I_a 
&= 
\langle x_* , M_+x_*  \rangle,\\
I_b
&=
- 2\langle x_* ,  \overline{\mathbf{w}}  \rangle -2 \sum_{j=m_++1}^m |c_j|\langle B_j x_*, \mathbf{v}_j\rangle   , \\
I_c
&= 
\langle M^{-1}\overline{\mathbf{w}},  \overline{\mathbf{w}}  \rangle + \sum_{j=m_++1}^m |c_j| \langle A_j^{-1}\mathbf{v}_j, \mathbf{v}_j\rangle.
\end{align*}
From \eqref{e:GenChoice} we have
\begin{align*}
I_a
&= 
\bigg\langle x_*, \sum_{j=1}^{m_+} c_j B_j^*  \mathbf{v}_j \bigg\rangle 
=
\sum_{j=1}^{m_+} c_j \langle B_j x_*, \mathbf{v}_j\rangle 
=
\sum_{j=1}^{m_+} c_j \langle A_j^{-1}\mathbf{v}_j, \mathbf{v}_j\rangle 
\end{align*}
and 
\begin{align*}
I_b =
-2 \bigg\langle x_*, \sum_{j=1}^{m_+} c_j B_j^*\mathbf{v}_j \bigg\rangle = 
-2\sum_{j=1}^{m_+} c_j \langle B_j x_*, \mathbf{v}_j\rangle 
=
-2 \sum_{j=1}^{m_+} c_j \langle A_j^{-1}\mathbf{v}_j, \mathbf{v}_j\rangle.
\end{align*}
So, putting things together, $I_a+I_b+I_c$ clearly coincides with the left-hand side of \eqref{e:GeneId}.
\end{proof}

In order to make rigorous the formal considerations in our above sketch proof of Theorem \ref{t:Geo}, it is important to identify an appropriately nice class of functions $f_j$ on which we may reduce matters. Similar considerations will also be necessary for our proof of Theorem \ref{t:QMain} in Section \ref{section:MainProof} and the content of Section \ref{subsection:nicef} below has been prepared primarily for this purpose.

\subsection{A log-convexity estimate}

Taking trace on both sides of the geometric condition $\sum_{j=1}^m c_j B_j^*B_j = \mathrm{id}$, we see that $\sum_{j=1}^m c_jn_j = n$ follows. In fact, in the case $\mathbf{G} = (\infty,\ldots,\infty)$ and $\mathcal{Q} = 0$, the condition $\sum_{j=1}^m c_jn_j = n$ is easily seen to be necessary for the strict positivity of the inverse Brascamp--Lieb constant by standard scaling considerations. For general $\mathbf{G}$, such a scaling condition is no longer a requirement and it is natural to modify the quantity $\mathfrak{Q}$ considered in \eqref{e:GeoQ} and mitigate for the loss of scaling by considering quantities of the form
\[
\mathfrak{Q}(t) = t^{-\alpha} \int_{\mathbb{R}^n} \prod_{j=1}^m u_j(t,B_jx)^{c_j} \, dx,
\]
where $\alpha := \frac{1}{2}(n - \sum_{j=1}^m c_jn_j)$. In such a case, the heat-flow monotonicity argument will make important use of the following log-convexity estimate.
\begin{lemma}[Log-convexity] \label{l:LiYau}
Let $G \in S_+(\mathbb{R}^n)$ and $u:\mathbb{R}^n \to (0,\infty)$ be of type $G$. Then
\begin{equation*}
D^2(\log u) \geq -2\pi G.
\end{equation*}
\end{lemma}
Obviously $D^2(\log g_G) = -2\pi G$ and thus Lemma \ref{l:LiYau} is equivalent to the log-convexity of the ratio $u/g_G$.
For a proof of Lemma \ref{l:LiYau}, we refer the reader to \cite[Lemma 8.6]{BCCT}. We also remark that estimates of the type in Lemma \ref{l:LiYau} are also known as \emph{Li--Yau gradient estimates} and hold much more generally in the framework of Riemannian manifolds; see \cite{LiYau}.

\subsection{The key decomposition of $\mathcal{Q}$}
In order to handle the gaussian kernel, our proof of Theorem \ref{t:QMain} is underpinned by the following decomposition of the transformation $\mathcal{Q}$ into positive and negative parts. 
\begin{proposition}[\cite{BW}] \label{p:BWdecomp}
	The transformation $\mathcal{Q} \in S(\mathbb{R}^n)$ satisfies \eqref{e:NondegBW} if and only if there exist linear surjections $B_0:\mathbb{R}^n \to \mathbb{R}^{n_0}$ and $B_{m+1}:\mathbb{R}^n \to \mathbb{R}^{n_{m+1}}$, and $\mathcal{Q}_+\in S_+(\mathbb{R}^{n_0})$ and $\mathcal{Q}_-\in S_+(\mathbb{R}^{n_{m+1}})$ such that 
\begin{equation}\label{e:DecompQ}
\mathcal{Q} = B_0^*\mathcal{Q}_+ B_0 - B_{m+1}^*\mathcal{Q}_-  B_{m+1},
\end{equation} 
the transformation $(B_0,\mathbf{B}_+) : \mathbb{R}^n \to \oplus_{j=0}^{m_+} \mathbb{R}^{n_j}$ is bijective, and ${\rm ker}\, \mathbf{B}_+ \subseteq {\rm ker}\, B_{m+1}$.
\end{proposition}
The above decomposition is due to Barthe and Wolff and we refer the reader to \cite[Lemma 3.1]{BW} for a proof.  We note here that, assuming the non-degeneracy condition \eqref{e:NondegBW}, Barthe and Wolff take $B_0$ to be the projection onto $\ker \mathbf{B}_+$ and we understand that in the case $\ker \mathbf{B}_+ = \{0\}$ we have $n_0 = 0$ and \eqref{e:DecompQ} becomes $\mathcal{Q} = - B_{m+1}^*\mathcal{Q}_-  B_{m+1}$ (similarly when $n_{m+1} = 0$ we understand that $\mathcal{Q} = B_0^*\mathcal{Q}_+B_0$). As we shall see in the forthcoming proof of Theorem \ref{t:QMain}, the decisive scenario is when $\mathbf{B}_+$ is bijective.

\subsection{Nice classes of input functions} \label{subsection:nicef}
The main purpose of this remaining part of Section \ref{section:prelims} is to identify appropriate classes of test functions which approximate general inputs and are sufficiently well behaved in order to facilitate, as far as possible, a proof of Theorem \ref{t:QMain} which is free from burdensome technicalities.

\begin{definition}
The class $\mathcal{N}$ is defined to be those inputs $\mathbf{f} = (f_j)_{j=1}^m$ satisfying the conditions
\begin{align}
f_j(x_j) \leq C_\mathbf{f} \1_{|x_j| \leq C_\mathbf{f}} \qquad &(1\le j \le m_+), \label{e:Linftyc}\\
f_j(x_j) \geq C_\mathbf{f}^{-1} (1+ |x_j|^2)^{-n_j} \qquad &(m_++1\le j \le m) \label{e:slowdecay},
\end{align}
for some constant $C_\mathbf{f} \in (0,\infty)$.
\end{definition}

\begin{lemma}\label{l:Approx1}
For any nondegenerate Brascamp--Lieb datum $(\mathbf{B},\mathbf{c},\mathcal{Q})$, we have
\[
\mathrm{I}(\mathbf{B},\mathbf{c},\mathcal{Q}) = \inf_{\mathbf{f} \in \mathcal{N}} \mathrm{BL}(\mathbf{B},\mathbf{c},\mathcal{Q};\mathbf{f}).
\]
\end{lemma}
\begin{proof}
For $N \geq 1$ and an arbitrary input $\mathbf{f}=(f_j)_{j=1}^m$, define the input $\mathbf{f}^{(N)}$ by
\begin{equation*}
f_j^{(N)}(x_j) :=
\left\{ 
\begin{array}{ll}
\1_{|x_j| \leq N}  \1_{f_j\le N}(x_j)  f_j(x_j) \quad & (1\le j \le m_+), \vspace{2mm}  \\ 
N^{-1} (1+|x_j|^2)^{-n_j} +  f_j(x_j)  \quad & (m_++1\le j \le m).
  \end{array} 
\right.
\end{equation*}
Clearly $\mathbf{f}^{(N)} \in \mathcal{N}$. Also, it is obvious that $f_j^{(N)} \uparrow f_j$ for $1\le j \le m_+$ and $f_j^{(N)} \downarrow f_j$ for $m_++1\le j \le m$ as $N\to \infty$. 
For $1\le j \le m_+$, an application of the monotone convergence theorem shows 
$
\int f_j^{(N)} \to \int f_j. 
$
For $m_++1\le j \le m$, since $(1+|x_j|^2)^{-n_j} \in L^1(\mathbb{R}^{n_j})$, we have
\begin{align*}
\lim_{N\to\infty} \int_{\mathbb{R}^{n_j}} f_j^{(N)}
&= 
\lim_{N\to\infty} \frac1N \int_{\mathbb{R}^{n_j}} (1+|x_j|^2)^{-n_j}\, dx_j + \int_{\mathbb{R}^{n_j}}  f_j 
=
\int_{\mathbb{R}^{n_j}} f_j. 
\end{align*}
Also, since 
$
(f_j^{(N)})^{c_j} \uparrow f_j^{c_j} 
$
for each $j=1,\ldots, m$, another application of the monotone convergence theorem gives
\begin{align*}
\lim_{N\to\infty} 
\int_{\mathbb{R}^n} e^{-\pi \langle x, \mathcal{Q} x\rangle} \prod_{j=1}^m f_j^{(N)}(B_jx)^{c_j}  \, dx
=
\int_{\mathbb{R}^n} e^{-\pi \langle x, \mathcal{Q} x\rangle} \prod_{j=1}^m f_j(B_jx)^{c_j} \, dx.
\end{align*}
The above shows
\[
\mathrm{BL}(\mathbf{B},\mathbf{c},\mathcal{Q};\mathbf{f}) = \lim_{N \to \infty} \mathrm{BL}(\mathbf{B},\mathbf{c},\mathcal{Q};\mathbf{f}^{(N)}) \geq  \inf_{\mathbf{\widetilde{f}} \in \mathcal{N}} \mathrm{BL}(\mathbf{B},\mathbf{c},\mathcal{Q};\mathbf{\widetilde{f}})
\]
and thus 
$
\mathrm{I}(\mathbf{B},\mathbf{c},\mathcal{Q}) \geq \inf_{\mathbf{f} \in \mathcal{N}} \mathrm{BL}(\mathbf{B},\mathbf{c},\mathcal{Q};\mathbf{f}).
$
The reverse inequality is trivial and thus we conclude the proof of the lemma.
\end{proof}

Thanks to the previous approximation lemma, we can deduce the following result. 
\begin{lemma}\label{l:limit}
If $(\mathbf{B}, \mathbf{c},\mathcal{Q},\mathbf{G})$ is a non-degenerate generalized Brascamp--Lieb datum, then
$$
\lim_{\lambda\to\infty}  {\rm I}(\mathbf{B},\mathbf{c},\mathcal{Q},\lambda \mathbf{G}) = {\rm I}(\mathbf{B},\mathbf{c},\mathcal{Q}),\;\;\; \lim_{\lambda\to\infty}  {\rm I}_{\mathrm{g}}(\mathbf{B},\mathbf{c},\mathcal{Q},\lambda \mathbf{G}) = {\rm I}_{\mathrm{g}}(\mathbf{B},\mathbf{c},\mathcal{Q}).
$$
\end{lemma}

\begin{remark}
Lemma \ref{l:limit} allows us to show that Theorem \ref{t:QMain} recovers Theorem \ref{t:BWQ} (and justifies the notation \eqref{Def:Typeinfty}). To deduce Theorem \ref{t:BWQ}, we take $(\mathbf{B},\mathbf{c},\mathcal{Q})$ such that \eqref{e:NondegBW} holds and use \cite[Proposition 2.2]{BW} to obtain the existence of $\lambda_0 > 0$ such that
\[
\mathcal{Q} + \lambda_0 \sum_{j=1}^{m_+} c_j B_j^*B_j > 0.
\] 
Thus, setting $G_j := \lambda_0\mathrm{id}$ for each $j=1,\ldots,m$, it follows that $(\mathbf{B},\mathbf{c},\mathcal{Q},\lambda \mathbf{G})$ is non-degenerate for all $\lambda \geq 1$. By using Lemma \ref{l:limit} and Theorem \ref{t:QMain} we obtain
\begin{align*}
\mathrm{I}(\mathbf{B},\mathbf{c},\mathcal{Q}) & = \lim_{\lambda \to \infty} {\rm I}(\mathbf{B},\mathbf{c},\mathcal{Q},\lambda \mathbf{G}) \\
& = \lim_{\lambda \to \infty} {\rm I}_{\mathrm{g}}(\mathbf{B},\mathbf{c},\mathcal{Q},\lambda \mathbf{G}) = \mathrm{I}_{\mathrm{g}}(\mathbf{B},\mathbf{c},\mathcal{Q})
\end{align*}
and we recover Theorem \ref{t:BWQ}.
\end{remark}

\begin{proof}[Proof of Lemma \ref{l:limit}]
To see $\lim_{\lambda\to\infty}  {\rm I}(\mathbf{B},\mathbf{c},\mathcal{Q},\lambda \mathbf{G}) = {\rm I}(\mathbf{B},\mathbf{c},\mathcal{Q})$, it clearly suffices to show
\begin{equation}\label{e:Goal23Mar}
\lim_{\lambda \to \infty}{\rm I}(\mathbf{B},\mathbf{c},\mathcal{Q},\lambda \mathbf{G}) \le {\rm I}(\mathbf{B},\mathbf{c},\mathcal{Q}).
\end{equation}
To see this, we  argue that
\[
\lim_{\lambda \to \infty}{\rm I}(\mathbf{B},\mathbf{c},\mathcal{Q},\lambda \mathbf{G}) \leq \lim_{\lambda \to \infty}{\rm BL}(\mathbf{B},\mathbf{c},\mathcal{Q};g_{\lambda\mathbf{G}} \ast \mathbf{f}) = {\rm BL}(\mathbf{B},\mathbf{c},\mathcal{Q};\mathbf{f})
\]
for arbitrary $\mathbf{f} \in \mathcal{N}$ and use Lemma \ref{l:Approx1}; thus it suffices to show
\begin{align}\label{e:limitchange}
\lim_{\lambda\to\infty} \int_{\mathbb{R}^n} e^{-\pi \langle x, \mathcal{Q} x\rangle} \prod_{j=1}^m g_{\lambda G_j}\ast f_j(B_jx)^{c_j}   \, dx 
= 
\int_{\mathbb{R}^n}  e^{-\pi \langle  x, \mathcal{Q} x\rangle} \prod_{j=1}^m  f_j(B_jx)^{c_j}   \, dx.
\end{align}
To see this, we use the fact that $f \in \mathcal{N}$ to get the bounds
\begin{align*}
g_{\lambda G_j } \ast f_j(x_j) & \leq C e^{-\pi \langle x_j, \frac{\lambda}{4} G_j x_j\rangle} \qquad (1\le j \le m_+) \\
g_{\lambda G_j } \ast f_j(x_j) &\ge C^{-1} (1+|x_j|^2)^{-n_j} \qquad  (m_++1\le j \le m)  
\end{align*}
for all $x_j \in \mathbb{R}^{n_j}$, and where $C \in (0,\infty)$ is a constant depending on $G_j$ and $C_\mathbf{f}$. Then
\begin{align*}
e^{-\pi \langle x, \mathcal{Q} x\rangle} \prod_{j=1}^m g_{\lambda G_j}\ast f_j(B_jx)^{c_j}  \label{e:6May3} \leq C e^{-\pi \langle x, (\mathcal{Q} + \frac{\lambda}{4}M_+) x\rangle}
 (1+|x|^2)^{N}
\end{align*}
for all $x \in \mathbb{R}^n$, where $M_+:= \sum_{j=1}^{m_+} c_jB_j^*G_jB_j$, $N := \sum_{j=m_++1}^m |c_j|n_j$, and $C$ is a finite constant depending on $\mathbf{B}, \mathbf{G}$ and $\mathbf{f}$, but independent of $\lambda$. Thanks to the non-degeneracy condition \eqref{e:NondegG}, it follows that $\mathcal{Q} + \frac{\lambda}{4}M_+$ is positive definite for $\lambda \geq 4$. Hence we may apply the dominated convergence theorem to deduce  \eqref{e:limitchange}.

Similarly, we can prove $\lim_{\lambda\to\infty}  {\rm I}_{\mathrm{g}}(\mathbf{B},\mathbf{c},\mathcal{Q},\lambda \mathbf{G}) = {\rm I}_{\mathrm{g}}(\mathbf{B},\mathbf{c},\mathcal{Q})$.
\end{proof}

Now we introduce $\mathbf{G} = (G_j)_{j=1}^m$, where $G_j \in S_+(\mathbb{R}^{n_j})$, and assume the generalized Brascamp--Lieb datum $(\mathbf{B},\mathbf{c},\mathbf{G},\mathcal{Q})$ is non-degenerate. The following class of functions will play a prominent role in the heat flow proof of Theorem \ref{t:QMain}. 
\begin{definition}
The class $\mathcal{N}(\mathbf{G})$ is defined to be those inputs $\mathbf{f} = (f_j)_{j=1}^m \in \mathcal{T}(\mathbf{G})$ of the form
$
f_j = g_{G_j} * h_j
$
with $h_j \in \mathcal{N}$, $j=1,\ldots,m$.
\end{definition}
Regarding this function class, we first note that the following analogue of Lemma \ref{l:Approx1} holds.
\begin{lemma} \label{l:ApproxG}
We have
\[
\mathrm{I}(\mathbf{B},\mathbf{c},\mathcal{Q},\mathbf{G}) = \inf_{\mathbf{f} \in \mathcal{N(\mathbf{G})}} \mathrm{BL}(\mathbf{B},\mathbf{c},\mathcal{Q};\mathbf{f}).
\]
\end{lemma}
One can prove Lemma \ref{l:ApproxG} in a similar manner to the proof of Lemma \ref{l:Approx1} and so we omit the details.

Although we considered isotropic heat flow \eqref{e:isoheat} for geometric data, for more general data it will be appropriate to consider certain anisotropic heat flows. Associated with such flows, we have the following pointwise estimates whose role will be to make rigorous the forthcoming heat-flow proof of Theorem \ref{t:QMain} for data such that gaussian maximizers exist.
\begin{lemma} \label{l:heatdecay}
Let $(\mathbf{B},\mathbf{c})$ be a Brascamp--Lieb datum. Suppose $\mathbf{f}\in \mathcal{N}(\mathbf{G})$ and $\mathbf{A} \leq \mathbf{G}$. For each $j=1,\ldots,m$, suppose further that $u_j : (1,\infty) \times \mathbb{R}^{n_j} \to (0,\infty)$ is a solution to the heat equation
\begin{equation} \label{e:heateqn}
\partial_t {u}_j = \frac1{4\pi} {\rm div}\, (A_j^{-1}\nabla {u}_j ), \qquad {u}_j(1,x_j) = f_j(x_j).
\end{equation}
Fix $t > 1$ and $\varepsilon > 0$. Then 
\begin{equation} \label{e:vbound}
|\nabla \log u_j (t,B_jx)| \leq C_*(t)(1+|x|^2)^{n_j} \qquad (j = 1,\ldots,m),
\end{equation}
\begin{equation} \label{e:timebound}
|\partial_t \log u_j (t,B_jx)| \leq C_*(t)(1+|x|^2)^{\max\{2,n_j\}} \qquad (j = 1,\ldots,m),
\end{equation}
and
\begin{equation} \label{e:Ubound}
\prod_{j=1}^m u_j(t,B_jx)^{c_j}  \leq C_*(t,\varepsilon) (1 + |x|^2)^N e^{-(1-\varepsilon) \pi \langle  x,t^{-1}M_+x  \rangle}
\end{equation}
for all $x \in \mathbb{R}^n$. Here, $M_+ := \sum_{j=1}^{m_+} c_j B_j^*A_jB_j$ and $N := \sum_{j=m_++1}^m |c_j|n_j$. Also, $C_*(t)$ denotes a strictly positive and finite constant which depends on $t$, $\mathbf{B}$, $\mathbf{c}$, $\mathbf{A}$, $\mathbf{G}$, $\mathbf{f}$ and is locally uniformly bounded in $t$, and $C_*(t,\varepsilon)$ denotes such a constant which also depends on $\varepsilon$.
\end{lemma}

\begin{proof}
Since $\mathbf{f} \in \mathcal{N}(\mathbf{G})$, we may write $f_j = g_{G_j} \ast h_j$, where $h_j \in \mathcal{N}$. For $j \in \{1,\ldots,m_+\}$ we let $R_j > 0$ be such that the support of $h_j$ is contained in the ball of radius $R_j$ with centre at the origin.

It is easily verified (using, say, the Fourier transform) that $u_j$ can be explicitly written down as 
\begin{equation} \label{e:heatsol}
u_j(t,x_j)
=
g_{P_j(t)} \ast h_j(x_j), 
\end{equation}
where $P_j(t) := (G_j^{-1} + (t-1) A_j^{-1})^{-1}$.

First we prove \eqref{e:Ubound}. If $j \in \{1,\ldots,m_+\}$, then one can easily check that
\begin{align*}
u_j(t,x_j) & = (\det P_j(t))^{1/2} \int_{\mathbb{R}^{n_j}} e^{-\pi \langle x_j -y_j , P_j(t)(x_j-y_j) \rangle} h_j(y_j) \, dy_j  \\
& \leq (\det P_j(t))^{1/2} \bigg( \int_{\mathbb{R}^{n_j}} h_j \bigg) e^{-(1-\varepsilon)\pi \langle x_j,P_j(t)x_j \rangle}
\end{align*}
for $|x_j| \geq \frac{1}{\varepsilon \pi}\|P(t)^{1/2}\| \| P(t)^{-1/2} \| R_j$. It follows that
\begin{equation*} 
u_j(t,x_j) \le C_*(t,\varepsilon) e^{-(1-\varepsilon)\pi \langle x_j,P_j(t)x_j\rangle}
\end{equation*}
for all $x_j \in \mathbb{R}^{n_j}$. By assumption we have $A_j\le G_j $ and as a consequence $P_j(t) \ge A_j/t$. Hence 
\begin{equation*}  
u_j(t,x_j) 
\le
C_*(t,\varepsilon) e^{-(1-\varepsilon)\pi \langle x_j, t^{-1}A_j x_j\rangle}
\end{equation*}
for all $x_j \in \mathbb{R}^{n_j}$.

For $j \in \{m_++1,\ldots,m\}$, we have
\begin{align*}
u_j(t,x_j) & \geq  \int_{\mathbb{R}^{n_j}} (1 + |x_j-P_j(t)^{-1/2}y_j|^2)^{-n_j} e^{-\pi |y_j|^2}  \, dy_j  \\
& \geq C_*(t)^{-1} (1 + |x_j|^2)^{-n_j} 
\end{align*}
for all $x_j \in \mathbb{R}^{n_j}$, where the second lower bound follows by restricting the domain of integration to $|y_j| \leq \|P_j(t)^{-1/2}\|^{-1}$. From the above, we clearly obtain \eqref{e:Ubound}.
 
Next we check \eqref{e:vbound}. For $j \in \{1,\ldots,m_+\}$, we use the compactness of the support of $h_j$ to obtain
\begin{align*}
|\nabla u_j (t,x_j)| & \leq 2\pi \|P_j(t)\| (\det P_j(t))^{1/2} \int_{\mathbb{R}^{n_j}} |x_j-y_j| e^{-\pi \langle x_j -y_j , P_j(t)(x_j-y_j) \rangle} h_j(y_j) \, dy_j  \\
& \leq C_*(t)(1 + |x_j|) u_j(t,x_j)
\end{align*}
for all $x_j \in \mathbb{R}^{n_j}$, and this clearly suffices for \eqref{e:vbound} for such $j$. For $j \in \{m_++1,\ldots,m\}$, we use the fact that $P_j(t) > 0$ to obtain $|w_j| e^{-\pi \langle w_j , P_j(t)w_j \rangle} \leq C_*(t)$ for all $w_j \in \mathbb{R}^{n_j}$, and therefore
$
|\nabla u_j (t,x_j)| \leq C_*(t)
$
for all $x_j \in \mathbb{R}^{n_j}$. It follows that
\[
|\nabla \log u_j (t,x_j)| \leq C_*(t)  (1 + |x_j|^2)^{n_j} 
\]
for all $x_j \in \mathbb{R}^{n_j}$. 

Finally we note that one can essentially follow the above argument for \eqref{e:vbound} in order to show \eqref{e:timebound} and so we omit the details.
\end{proof}
We shall also need the large time asymptotics of the solution to \eqref{e:heateqn}. From the explicit form of the solution \eqref{e:heatsol} we easily see that
\begin{equation} \label{e:largetime}
\lim_{t \to \infty} t^{n_j/2} u_j(t,\sqrt{t}x_j) = \bigg( \int_{\mathbb{R}^{n_j}} f_j \bigg) g_{A_j}(x_j)
\end{equation}
for each $x_j \in \mathbb{R}^{n_j}$.

\section{Closure properties of sub/supersolutions of heat equations} \label{section:cp}

The heat-flow proof of Theorem \ref{t:Geo} in the previous section rested on the non-increasingness of 
\[
\mathfrak{Q}(t) = \int_{\mathbb{R}^n} U(t,x) \, dx,
\]
where $U$ is the ``anisotropic geometric mean" given by
\[
U(t,x) = \prod_{j=1}^m (u_j \circ B_j)^{c_j}
\]
and $u_j$ is a solution to the heat equation $\partial_t u_j = \Delta u_j$ with non-negative initial data. Although it is not immediate from the argument we presented in the previous section, one can show that $U$ satisfies $\partial_t U \leq \Delta U$; that is, $U$ is a \emph{subsolution} of the corresponding heat equation. Formally at least, the monotonicity of $\mathfrak{Q}$ can be seen in this way since
\begin{equation} \label{e:mono_formal}
\mathfrak{Q}'(t) =  \int_{\mathbb{R}^n} \partial_t U(t,x) \, dx \leq  \int_{\mathbb{R}^n} \Delta U(t,x) \, dx = 0.
\end{equation}
In fact, in the case of geometric Brascamp--Lieb data, it is more generally true that if $u_j$ are non-negative \emph{subsolutions} for $j=1,\ldots,m_+$ and non-negative \emph{supersolutions} for $j=m_++1,\ldots,m$, then $\partial_t U \leq \Delta U$; in other words
\begin{equation} \label{e:cp_Igeo}
c_j(\partial_t u_j - \Delta u_j) \leq 0 \quad (j=1,\ldots,m) \qquad \Rightarrow \qquad \partial_t U - \Delta U \leq 0.
\end{equation}
This observation is a kind of reverse counterpart to the observation, or ``closure property", 
\begin{equation} \label{e:cp_Fgeo}
\partial_t u_j - \Delta u_j \geq 0 \quad (j=1,\ldots,m) \qquad \Rightarrow \qquad \partial_t U - \Delta U \geq 0
\end{equation}
recently presented in \cite{BBCrell}; in the manner described above, the closure property \eqref{e:cp_Fgeo} generates the forward version of the geometric Brascamp--Lieb inequality (i.e. $\mathrm{F}(\mathbf{B},\mathbf{c}) = 1$ for geometric Brascamp--Lieb data). The perspective taken in the article \cite{BBCrell} is that it is natural to place oneself in the framework of sub/supersolutions to heat equations since collections of sub/supersolutions are closed under a wide variety of operations (and this is not the case for \emph{bona fide} solutions). As a result, one is able to generate, in a \emph{systematic manner} by combining various closure properties, monotone quantities which underpin a number of fundamental inequalites in geometric analysis and neighbouring fields. 

In this section, we build on \cite{BBCrell} and present a new closure property which generalizes \eqref{e:cp_Igeo} to general Brascamp--Lieb data; later we apply our closure property as a key step in our proof of Theorem \ref{t:QMain}. 

For any Brascamp--Lieb datum $(\mathbf{B},\mathbf{c},\mathcal{Q})$ and gaussian input $\mathbf{A} = (A_j)_{j=1}^m$, with $A_j \in S_+(\mathbb{R}^{n_j})$, we write
\[
M(\mathbf{B},\mathbf{c};\mathbf{A}) := \sum_{j=1}^m c_j B_j^*A_jB_j
\]
and
\[ 
\widetilde{M}(\mathbf{B},\mathbf{c},\mathcal{Q};\mathbf{A}) := M(\mathbf{B},\mathbf{c};\mathbf{A})+\mathcal{Q}.
\]
When there is no danger of any confusion, we shall often omit the dependence on the data and simply write $M(\mathbf{A})$ and $\widetilde{M}(\mathbf{A})$ (sometimes we may also drop the dependence on $\mathbf{A}$).
\begin{theorem} \label{t:cp_IBL}
Let $(\mathbf{B},\mathbf{c},\mathbf{G},\mathcal{Q})$ be a non-degenerate  generalized Brascamp--Lieb datum and suppose $\widetilde{M}(\mathbf{A})$ is positive definite. Assume further that
\begin{equation} \label{e:closurekey1}
c_j(A_j^{-1}  - B_j\widetilde{M}(\mathbf{A})^{-1}B_j^*) \leq 0 
\end{equation}
and
\begin{equation} \label{e:closurekey2}
(A_j^{-1} - B_j\widetilde{M}(\mathbf{A})^{-1}B_j^*)(G_j-A_j) = 0
\end{equation}
for $j=1,\ldots,m$. Given $u_j : (1,\infty) \times \mathbb{R}^{n_j} \to (0,\infty)$ for $j=1,\ldots,m$, let $U : (1,\infty) \times \mathbb{R}^n \to (0,\infty)$ be given by
\begin{equation*}
U(t,x) := t^{-\alpha} \prod_{j=0}^{m+1} u_j(t,B_jx)^{c_j}.
\end{equation*}
Here
$
\alpha := \frac{1}{2}(n - \sum_{j=0}^{m+1} c_jn_j),
$
$c_0=1$, $c_{m+1}=-1$, and $u_0$ and $u_{m+1}$ are given by  
\begin{align*}
	\partial_t u_0 &= \frac1{4\pi} {\rm div}\, (\mathcal{Q}_+^{-1} \nabla u_0),\;\;\; u_0(1) = g_{\mathcal{Q}_+},\\
	\partial_t u_{m+1} &= \frac1{4\pi} {\rm div}\, (\mathcal{Q}_-^{-1} \nabla u_{m+1}),\;\;\; u_{m+1}(1) = g_{\mathcal{Q}_-}. 
\end{align*}

If
\begin{equation} \label{e:u_jsub}
c_j(\partial_t u_j - \frac{1}{4\pi} {\rm div}  (A_j^{-1} \nabla u_j) ) \leq 0 \quad \text{and} \quad D^2(\log u_j) \geq -\frac{2\pi G_j}{t}
\end{equation}
for $j=1,\ldots,m$, then
\begin{equation} \label{e:closuregoal}
\partial_t U - \frac{1}{4\pi} {\rm div}  (\widetilde{M}(\mathbf{A})^{-1} \nabla U) \leq 0. 
\end{equation}
\end{theorem}

\begin{remarks} (1) When we apply Theorem \ref{t:cp_IBL}, $\mathbf{A}$ will be such that the gaussian input $(g_{A_j})_{j=1}^m$ minimizes the inverse Brascamp--Lieb constant over all gaussian inputs of type $\mathbf{G}$. For such a minimizer, we shall see (in Lemma \ref{l:mono_key}) that conditions of the form \eqref{e:closurekey1} and \eqref{e:closurekey2} are satisfied.

(2) Suppose $u_j : (1,\infty) \times \mathbb{R}^{n_j} \to (0,\infty)$ is a \emph{solution} to the heat equation
\begin{equation*}
\partial_t {u}_j = \frac1{4\pi} {\rm div}\, (A_j^{-1}\nabla {u}_j ), \qquad {u}_j(1,x_j) = f_j(x_j),
\end{equation*}
where $f_j \in \mathcal{T}(G_j)$.
Then we know from \eqref{e:heatsol} that $u_j \in \mathcal{T}((G_j^{-1} + (t-1) A_j^{-1})^{-1})$. Thus, if $A_j \leq G_j$ then we have $u_j \in \mathcal{T}(t^{-1}G_j)$ and it follows from Lemma \ref{l:LiYau} that $D^2(\log u_j) \geq -2\pi t^{-1} G_j$. In this sense, the log-convexity component in the assumption \eqref{e:u_jsub} is reasonable. At the end of this subsection, we add a few further remarks on this.

(3) Theorem \ref{t:cp_IBL} is very much in the spirit of \cite[Theorem 3.7]{BBCrell} in which it is shown that $U$ is a supersolution in the case $\mathcal{Q}=0$, $c_j > 0$, and $A_j = G_j$ for each $j$; that is, in such a setting, if $A_j^{-1}  - B_jM(\mathbf{A})^{-1}B_j^* \geq 0$ for each $j$, then we have $\partial_t U - \frac{1}{4\pi} {\rm div}  (M(\mathbf{A})^{-1} \nabla U) \geq 0$ under the assumption that 
\begin{equation*}
\partial_t u_j - \frac{1}{4\pi} {\rm div}  (A_j^{-1} \nabla u_j) \geq 0 \quad \text{and} \quad D^2(\log u_j) \geq -\frac{2\pi A_j}{t}.
\end{equation*}
In fact, \cite[Theorem 3.7]{BBCrell} also provides the conclusion that $U$ obeys the log-convexity inequality
\[
D^2(\log U) \geq -\frac{2\pi M(\mathbf{A})}{2t},
\]
which can be easily deduced from the identity
\begin{equation} \label{e:hessian}
D^2(\log U) = \sum_{j=1}^m c_j B_j^* D^2(\log u_j) B_j
\end{equation}
and the positivity of each $c_j$. In the setting of mixed signatures for the $c_j$ in Theorem \ref{t:cp_IBL}, a closure property related  to the log-convexity assumption seems less apparent.
\end{remarks}

\begin{proof}[Proof of Theorem \ref{t:cp_IBL}]
For simplicity, we write $\widetilde{M}= \widetilde{M}(\mathbf{A})$ in this proof. 
Let us denote $A_0 := \mathcal{Q}_+$ and $A_{m+1} := \mathcal{Q}_-$ in which case the decomposition \eqref{e:DecompQ} can be seen as $\mathcal{Q} = c_0 B_0^*A_0 B_0 +c_{m+1} B_{m+1}^* A_{m+1}B_{m+1}$.  
By straightforward computations
\begin{equation} \label{e:Utime}
\frac{\partial_t U}{U}(t,x) = -\frac{\alpha}{t} + \sum_{j=0}^{m+1} c_j \frac{\partial_tu_j}{u_j}(t,B_jx).
\end{equation}
Similarly, we have
\begin{equation} \label{e:Uspace}
\nabla U(t,x) = U(t,x) \sum_{j=0}^{m+1} c_j B_j^*\mathbf{v}_j(t,B_jx),
\end{equation}
where $\mathbf{v}_j:= \nabla \log u_j$, from which it follows that
\[
\frac{ {\rm div}  (\widetilde{M}^{-1} \nabla U)}{U}(t,x) = \sum_{j=0}^{m+1} c_j {\rm div}(B_j\widetilde{M}^{-1}B_j^*\mathbf{v}_j)  + \langle \overline{\mathbf{w}},\widetilde{M}^{-1}\overline{\mathbf{w}}\rangle, 
\]
where $\overline{\mathbf{w}} := \sum_{j=0}^{m+1} c_j B_j^*\mathbf{v}_j$. Here, and in what follows, on the right-hand side we are suppressing the argument of the functions; precisely speaking, we should write $u_j(t,B_jx)$, $\mathbf{v}_j(t,B_jx)$ and $\overline{\mathbf{w}}(t,x)$.

From \eqref{e:u_jsub} and definitions of $u_0, u_{m+1}$, we have that 
\begin{align*}
\frac{\partial_t U}{U}(t,x) & \leq -\frac{\alpha}{t} + \frac{1}{4\pi} \sum_{j=0}^{m+1} c_j \frac{{\rm div}(A_j^{-1}\nabla u_j)}{u_j} \\
& =  -\frac{\alpha}{t} + \frac{1}{4\pi} \sum_{j=0}^{m+1} c_j  {\rm div} (A_j^{-1} \mathbf{v}_j)  +  \frac{1}{4\pi}\sum_{j=0}^{m+1} c_j \langle \mathbf{v}_j, A_j^{-1} \mathbf{v}_j  \rangle
\end{align*}
and therefore
\[
U^{-1}\bigg(\partial_t U -  \frac{1}{4\pi}{\rm div}  (\widetilde{M}^{-1} \nabla U)\bigg)(t,x) \leq \frac{1}{4\pi}(I + II),
\]
where
\begin{align*}
I & :=  - \langle \overline{\mathbf{w}},\widetilde{M}^{-1}\overline{\mathbf{w}}\rangle  + \sum_{j=0}^{m+1} c_j \langle \mathbf{v}_j, A_j^{-1} \mathbf{v}_j  \rangle\\
II & :=  - \frac{4\pi \alpha}{t} + \sum_{j=0}^{m+1} c_j  {\rm div}((A_j^{-1} - B_j\widetilde{M}^{-1}B_j^*)\mathbf{v}_j).
\end{align*}
Since $(B_0,\mathbf{B}_+)$ is surjective from the non-degeneracy assumption,  we know from \eqref{e:GeneKey} in Lemma \ref{l:Key2} that $I \leq 0$.
For $II$, we write
\[
II = - \frac{4\pi \alpha}{t} + \sum_{j=0}^{m+1}   {\rm tr}(c_j(A_j^{-1} - B_j\widetilde{M}^{-1}B_j^*) D^2 \log u_j).
\]
From \eqref{e:closurekey1}, \eqref{e:u_jsub} and \eqref{e:closurekey2}, it follows that {for $j=1,\ldots,m$,} 
\begin{align*}
 {\rm tr}(c_j(A_j^{-1} - B_j\widetilde{M}^{-1}B_j^*) D^2 \log u_j) & \leq  -\frac{2\pi}{t} {\rm tr}(c_j(A_j^{-1} - B_j\widetilde{M}^{-1}B_j^*) G_j) \\
& =  -\frac{2\pi c_j}{t} {\rm tr}((A_j^{-1} - B_j\widetilde{M}^{-1}B_j^*) A_j) \\
& = -\frac{2\pi c_j}{t} (n_j - {\rm tr}( \widetilde{M}^{-1}B_j^*A_jB_j) ).
\end{align*}
On the other hand, for $j=0,m+1$, from the explicit formula \eqref{e:heatsol} we know that 
$$
D^2(\log u_j) = -\frac{2\pi {A}_j}{t}.
$$
Hence,  {regardless of the sign of $A_j^{-1} - B_j \widetilde{M}^{-1}B_j^*$}, we have that 
\begin{align*} 
 {\rm tr}(c_j(A_j^{-1} - B_j\widetilde{M}^{-1}B_j^*) D^2 \log u_j) & {=}  -\frac{2\pi}{t} {\rm tr}(c_j(A_j^{-1} - B_j\widetilde{M}^{-1}B_j^*) A_j) \\
& = -\frac{2\pi c_j}{t} (n_j - {\rm tr}( \widetilde{M}^{-1}B_j^*A_jB_j) )
\end{align*}
for $j=0, m+1$. Therefore, from the definition of $\widetilde{M}$, we get
\begin{equation*}
II \leq - \frac{4\pi \alpha}{t} -\frac{2\pi}{t} \sum_{j=0}^{m+1} c_j   (n_j - {\rm tr}( \widetilde{M}^{-1}B_j^*A_jB_j) ) = 0. \qedhere
\end{equation*}
\end{proof}

We end this subsection with two further observations related to Remarks (2) and (3) after the statement of Theorem \ref{t:cp_IBL}. The first concerns the analogue of Theorem \ref{t:cp_IBL} for which $c_j > 0$ for all $j=1,\ldots,m$ (in which case we drop the non-degenerate condition and assume $\mathcal{Q}\ge0$, namely $n_{m+1}=0$). It is clear from the above proof (and \eqref{e:hessian}) that in such a case, if we assume
\[
A_j^{-1} - B_j\widetilde{M}(\mathbf{A})^{-1}B_j^* \geq 0
\]
and \eqref{e:closurekey2}, then
\begin{equation} \label{e:u_jsuper}
\partial_t u_j - \frac{1}{4\pi} {\rm div}  (A_j^{-1} \nabla u_j)  \geq 0 \quad \text{and} \quad D^2(\log u_j) \geq -\frac{2\pi G_j}{t}
\end{equation}
for $j=1,\ldots,m$ imply
\begin{equation} \label{e:closuregoal_super}
\partial_t U - \frac{1}{4\pi} {\rm div}  (\widetilde{M}(\mathbf{A})^{-1} \nabla U) \geq 0 \quad \text{and} \quad D^2(\log U) \geq -\frac{2\pi \widetilde{M}(\mathbf{G})}{2t}. 
\end{equation}
This observation extends \cite[Theorem 3.7]{BBCrell} to the setting $\mathbf{A} \leq \mathbf{G}$ and $\mathcal{Q}\ge0$.

Our second observation here concerns the log-convexity assumption in \eqref{e:u_jsub}, and we claim that it would be reasonable to assume
\begin{equation} \label{e:logconv_harmonic}
D^2(\log u_j) \geq -2\pi (G_j^{-1} + (t-1)A_j^{-1})^{-1}.
\end{equation}
Indeed, as we have noted several times, solutions of the heat equation \eqref{e:heateqn} with $u_j(1,\cdot) \in \mathcal{T}(G_j)$ satisfy $u_j(t,\cdot) \in \mathcal{T}((G_j^{-1} + (t-1)A_j^{-1})^{-1})$, and hence \eqref{e:logconv_harmonic} by Lemma \ref{l:LiYau}. We claim that, when $c_j > 0$ for all $j=1,\ldots,m$, then \eqref{e:logconv_harmonic} is also closed under the operation $(u_1,\ldots,u_m) \mapsto U$ in the sense that
\begin{equation} \label{e:logconv_HM_closed}
D^2(\log U) \geq -2\pi (\widetilde{M}(\mathbf{G})^{-1} + (t-1)\widetilde{M}(\mathbf{A})^{-1})^{-1}.
\end{equation}
To see this, following the notation in \cite{Hiai} we write
\[
X : Y := (X^{-1} + Y^{-1})^{-1}
\]
for the harmonic mean of the positive semi-definite transformations $X$ and $Y$, and note the fundamental facts (see, for example, \cite[Corollary 3.1.6]{Hiai}):

($\mathrm{I}$) $S^*(X : Y) S \leq (S^*XS) : (S^*YS)$. 

($\mathrm{II}$) $(X_1 : Y_1) + (X_2 : Y_2) \leq (X_1 + X_2) : (Y_1 + Y_2)$.

 If we first use \eqref{e:hessian} and \eqref{e:logconv_harmonic}, we get
\[
D^2(\log U) \geq -2\pi \sum_{j=0}^m c_j B_j^* (G_j : A_{j,t}) B_j
\]
where $c_0=1$, $A_0=G_0=\mathcal{Q}_+$ and $A_{j,t} := (t-1)^{-1}A_j$. However
\begin{align*}
 \bigg( \sum_{j=0}^m c_j B_j^* G_j B_j \bigg) : \bigg(  \sum_{j=0}^m c_j B_j^* A_{j,t} B_j \bigg) & \geq \sum_{j=0}^m (c_j B_j^*G_jB_j : c_j B_j^*A_{j,t}B_j) \\
& =  \sum_{j=0}^m c_j (B_j^*G_jB_j : B_j^*A_{j,t}B_j) \\
& \geq \sum_{j=0}^m c_j B_j^* (G_j : A_{j,t} )B_j,
\end{align*}
where we have successively applied ($\mathrm{II}$) (in an iterative way for sums of $m$ transformations), linearity, and ($\mathrm{I}$). This yields \eqref{e:logconv_HM_closed}.

We incorporate the preceding observation into the following (independent) result in the spirit of \cite[Theorem 3.7]{BBCrell}.
\begin{theorem}\label{t:Forwardcp} 
Let $(\mathbf{B},\mathbf{c})$ be a Brascamp--Lieb datum with $c_j > 0$ for all $j=1,\ldots,m$ and $\mathcal{Q} = B_0^*\mathcal{Q}_+B_0\ge0$. Suppose $\mathbf{A} = (A_j)_{j=1}^m, \mathbf{G} = (G_j)_{j = 1}^m$, with $A_j, G_j \in S_+(\mathbb{R}^{n_j})$, are such that $\widetilde{M}(\mathbf{A})$ and $\widetilde{M}(\mathbf{G})$ are positive definite. Assume further that
\begin{equation*} 
A_j^{-1}  - B_j\widetilde{M}(\mathbf{A})^{-1}B_j^* \geq 0 
\end{equation*}
and
\begin{equation*} 
(A_j^{-1} - B_j\widetilde{M}(\mathbf{A})^{-1}B_j^*)(G_j-A_j) = 0
\end{equation*}
for $j=1,\ldots,m$. Given $u_j : (1,\infty) \times \mathbb{R}^{n_j} \to (0,\infty)$ for $j=1,\ldots,m$, let $U : (1,\infty) \times \mathbb{R}^n \to (0,\infty)$ be given by
\begin{equation*}
U(t,x) := t^{-\alpha} \prod_{j=0}^m u_j(t,B_jx)^{c_j},
\end{equation*}
where
$
\alpha := \frac{1}{2}(n - \sum_{j=0}^m c_jn_j),
$
$c_0=1$, and $u_0$ be as in Theorem \ref{t:cp_IBL}. 
 If
\begin{equation*}
\partial_t u_j \geq \frac{1}{4\pi} {\rm div}  (A_j^{-1} \nabla u_j) )  \quad \text{and} \quad D^2(\log u_j) \geq -2\pi (G_j^{-1} + (t-1)A_j^{-1})^{-1}
\end{equation*}
for $j=1,\ldots,m$, then
\begin{equation*}
\partial_t U \geq \frac{1}{4\pi} {\rm div}  (\widetilde{M}(\mathbf{A})^{-1} \nabla U)  \quad \text{and} \quad D^2(\log U) \geq -2\pi (\widetilde{M}(\mathbf{G})^{-1} + (t-1)\widetilde{M}(\mathbf{A})^{-1})^{-1}.
\end{equation*}
\end{theorem}

\section{Proof Theorem \ref{t:QMain}} \label{section:MainProof}

The proof makes use of the key decomposition in Proposition \ref{p:BWdecomp}. 

\emph{Section \ref{section:B_0=0_heatflow}.} Inspired by ideas in \cite{BCCT}, we begin by establishing Theorem \ref{t:QMain} in the so-called ``gaussian extremizable" case (see Definition \ref{d:gaussian_ext} below for the precise definition) via a heat-flow monotonicity argument; see Theorem \ref{t:GkerGExt}. At this stage, we appeal to Theorem \ref{t:cp_IBL}. 

\emph{Section \ref{section:B_0=0_globaldata}.} In order to effectively reduce to the gaussian extremizable case, we introduce the notion of ``amplifying data" (Definition \ref{d:globalised_data}), for which gaussian extremizers always exist, and then complete the proof of Theorem \ref{t:QMain} by showing that arbitrary non-degenerate Brascamp--Lieb data can be approximated in an appropriate sense by amplifying data.

Before beginning the proof of Theorem \ref{t:QMain}, we recall the notation
\[
M(\mathbf{B},\mathbf{c};\mathbf{A}) := \sum_{j=1}^m c_jB_j^*A_jB_j
\]
for any Brascamp--Lieb datum $(\mathbf{B},\mathbf{c})$ and input $\mathbf{A} = (A_j)_{j=1}^m$, with $A_j \in  S_+(\mathbb{R}^{n_j})$.
We shall write $\mathbf{A}_+ := (A_j)_{j=1}^{m_+}$ and $\mathbf{A}_- := (A_j)_{j=m_++1}^{m}$, and similarly for $\mathbf{c}_\pm$ and $\mathbf{G}_\pm$. Using this notation, for example, the non-degeneracy condition \eqref{e:NondegG} becomes
\[
M(\mathbf{B}_+,\mathbf{c}_+;\mathbf{G}_+) + \mathcal{Q} > 0.
\]

Also it is natural to introduce the class 
\begin{equation}\label{Def:LambdaBcQ}
\Lambda(\mathbf{B},\mathbf{c},\mathcal{Q}) := \bigg\{ \mathbf{A} = (A_j)_{j=1}^m : A_j \in S_+(\mathbb{R}^{n_j}), M(\mathbf{B},\mathbf{c};\mathbf{A}) +\mathcal{Q} > 0 \bigg \}.
\end{equation}
For example, by definition, we have ${\rm BL}(\mathbf{B},\mathbf{c},\mathcal{Q}; \mathbf{A}) < \infty$ if and only if $\mathbf{A} \in \Lambda(\mathbf{B},\mathbf{c},\mathcal{Q})$. Also we note that the non-degenerate condition \eqref{e:NondegBW} ensures that $\Lambda(\mathbf{B},\mathbf{c},\mathcal{Q})$ is non-empty; see \cite[Proposition 2.2]{BW}. In fact, we have already used this observation in the remark after Lemma \ref{l:limit} in showing that Theorem \ref{t:QMain} implies Theorem \ref{t:BWQ}.

\subsection{Gaussian extremizable data} \label{section:B_0=0_heatflow}
First we introduce the definition of gaussian extremizable data and then proceed to show that Theorem \ref{t:QMain} holds for such data.
\begin{definition}[Gaussian extremizable data] \label{d:gaussian_ext}
	The generalized Brascamp--Lieb datum $(\mathbf{B},\mathbf{c},\mathbf{G},\mathcal{Q})$ is said to be \textit{gaussian extremizable} if 
	\begin{equation*}
	{{\rm I}_{\mathrm{g}}}(\mathbf{B},\mathbf{c},\mathcal{Q},\mathbf{G}) = {\rm BL}(\mathbf{B},\mathbf{c},\mathcal{Q};\mathbf{A}), 
	\end{equation*}
	for some $\mathbf{A} \in \Lambda(\mathbf{B},\mathbf{c},\mathcal{Q})$ with $ \mathbf{A} \leq \mathbf{G}$, and (with a slight abuse of terminology) we refer to such $\mathbf{A}$ as a \emph{gaussian extremizer}.
\end{definition}

\begin{theorem}\label{t:GkerGExt}
	Let $(\mathbf{B},\mathbf{c},\mathcal{Q},\mathbf{G})$ be a non-degenerate  generalized Brascamp--Lieb datum. Then the following are equivalent.

(1) $(\mathbf{B},\mathbf{c},\mathcal{Q},\mathbf{G})$ is gaussian extremizable. 

(2) There exists $\mathbf{A} \in \Lambda(\mathbf{B},\mathbf{c},\mathcal{Q})$ with $ \mathbf{A} \leq \mathbf{G}$ satisfying \eqref{e:closurekey1} and \eqref{e:closurekey2} for $j=1,\ldots,m$.  

(3) There exists $\mathbf{A} \in \Lambda(\mathbf{B},\mathbf{c},\mathcal{Q})$ with $ \mathbf{A} \leq \mathbf{G}$, such that 
		$$
		{\rm I}(\mathbf{B},\mathbf{c},\mathcal{Q},\mathbf{G}) = 
		{\rm I}_{\mathrm{g}}(\mathbf{B},\mathbf{c},\mathcal{Q},\mathbf{G})
		=
		{\rm BL}(\mathbf{B},\mathbf{c},\mathcal{Q};\mathbf{A}) \in (0,\infty).
		$$
\end{theorem}
The first thing to notice is that the implication \textit{(3)} $\Rightarrow$ \textit{(1)} is just a consequence of the definition of gaussian extremizability. Also, the implication \textit{(1)} $\Rightarrow$ \textit{(2)} is a consequence of the following. 
\begin{lemma}\label{l:mono_key}
	Suppose the non-degenerate generalized Brascamp--Lieb datum $(\mathbf{B}, \mathbf{c}, \mathcal{Q}, \mathbf{G})$ is gaussian extremizable. Then for any gaussian extremizer $\mathbf{A} \in \Lambda(\mathbf{B},\mathbf{c},\mathcal{Q})$ with $ \mathbf{A} \leq \mathbf{G}$ we have  
	\begin{equation} \label{e:key_matrixineq}
	c_j(A_j^{-1} - B_j \widetilde{M}^{-1} B_j^*) \le 0
	\end{equation}
	and   
	\begin{equation} \label{e:key_matrixeq}
	(A_j^{-1} - B_j \widetilde{M}^{-1} B_j^*)(G_j - A_j) = 0
	\end{equation}
	for all $j=1,\ldots,m$, where $\widetilde{M} := M(\mathbf{B},\mathbf{c};\mathbf{A}) + \mathcal{Q}$.
\end{lemma}

\begin{proof}[Proof of Lemma \ref{l:mono_key}]
We follow the basic strategy behind the argument in the proof of \cite[Corollary 8.11]{BCCT} with certain minor simplifications. We also remark that 
$\widetilde{M} > 0$ since $\mathbf{A} \in \Lambda(\mathbf{B},\mathbf{c},\mathcal{Q})$.	
	
The gaussian extremizability assumption implies
	\begin{equation}\label{e:localmax}
	\Phi(\mathbf{X}) \ge \Phi(\mathbf{A}) \quad \text{for all $\mathbf{X} \leq \mathbf{G}$},
	\end{equation}
	where	
	$$
	\Phi(\mathbf{X}) := \sum_{j=1}^{m} c_j \log\, {\rm \det}\, X_j - \log\, {\rm \det}\, \bigg( \sum_{j=1}^{m} c_j B_j^* X_j B_j +\mathcal{Q} 
	\bigg).
	$$
	We also have the identity 
	\begin{equation}\label{e:FormDeri}
	\frac{d}{d\varepsilon} \Phi(A_1,\ldots, A_j + \varepsilon D_j,\ldots, A_m)|_{\varepsilon=0+}
	= 
	c_j {\rm tr}\, ((A_j^{-1}- B_j \widetilde{M}^{-1} B_j^*) D_j)
	\end{equation}
	for any linear map $D_j : \mathbb{R}^{n_j} \to \mathbb{R}^{n_j}$, 
	which can easily be checked using the fact that $\log \det (I + \varepsilon A) = ({\rm tr} A) \varepsilon + O(\varepsilon^2)$.
	
	Now we fix $j \in \{1,\ldots,m\}$. Since $A_j \leq G_j$, for arbitrary $N_j \leq 0$ we clearly have
	\begin{equation*}\label{e:varepsilon0}
	0 < A_j + \varepsilon N_j \le G_j \quad \text{for all sufficiently small $\varepsilon>0$.}
	\end{equation*}
	Thus it follows from \eqref{e:localmax} that
	\begin{equation}\label{e:DeriPosi}
	\frac{d}{d\varepsilon} \Phi(A_1,\ldots, A_j + \varepsilon N_j,\ldots, A_m)|_{\varepsilon=0+}  \ge 0.
	\end{equation}
	From \eqref{e:FormDeri} we obtain
	\begin{equation} \label{e:Qjneg}
	{\rm tr}\, (c_j(A_j^{-1}- B_j \widetilde{M}^{-1} B_j^*) N_j) \geq 0 \quad \text{for all $N_j \leq 0$}
	\end{equation}
	and \eqref{e:key_matrixineq} follows. 
	
	To show \eqref{e:key_matrixeq}, we write 
	\begin{align*}
	P_j & := -c_j (A_j^{-1} - B_j \widetilde{M}^{-1} B_j^*) \\
	Q_j & := G_j-A_j
	\end{align*}
	and we begin with the seemingly weaker claim 
	\begin{equation}\label{e:Trace=0}
	{\rm tr}\, (P_jQ_j ) = 0.
	\end{equation}
	To see \eqref{e:Trace=0}, since $Q_j \ge 0$ we clearly have
	$$
	0< A_j + \varepsilon Q_j \le G_j \quad \text{for all $\varepsilon \in (0,1)$.}
	$$
	Therefore, from \eqref{e:localmax} and \eqref{e:FormDeri} we have
	$$
	{\rm tr}\, (P_j Q_j) \leq 0.
	$$
	On the other hand, we may apply \eqref{e:Qjneg} with $N_j = - Q_j$ to see 
	$$
	{\rm tr}\, (P_j Q_j) \geq 0
	$$
	and hence \eqref{e:Trace=0}. 
	
	To obtain \eqref{e:key_matrixeq} from \eqref{e:Trace=0}, we note that the cyclic property of the trace gives ${\rm tr}\, (R_j^* R_j) = 0$, where
	$
	R_j:=  P_j^{1/2} Q_j^{1/2}.
	$
	This means $R_j = 0$ and hence 
	$
	P_j Q_j = P_j^{1/2} R_j Q_j^{1/2} = 0,
	$
	concluding our proof of \eqref{e:key_matrixeq}.
\end{proof}

\begin{proof}[Proof of Theorem \ref{t:GkerGExt}]
With Lemma \ref{l:mono_key} in mind, to prove Theorem \ref{t:GkerGExt} it remains to show the implication \textit{(2)} $\Rightarrow$ \textit{(3)}, and thus we assume \textit{(2)}. Making use of the decomposition \eqref{e:DecompQ} of $\mathcal{Q}$, given an arbitrary $\mathbf{A} \in \Lambda(\mathbf{B},\mathbf{c},\mathcal{Q})$ with $\mathbf{A} \leq \mathbf{G}$ satisfying \eqref{e:closurekey1} and \eqref{e:closurekey2}, it suffices by Lemma \ref{l:ApproxG} to show 
\begin{equation}\label{e:flow_goal}
\int_{\mathbb{R}^n } \prod_{j=0}^{m+1} f_j(B_jx)^{c_j} \, dx \ge {\rm BL}(\mathbf{B},\mathbf{c},\mathcal{Q}; \mathbf{A}) \prod_{j=1}^m\bigg( \int_{\mathbb{R}^{n_j}} f_j\bigg)^{c_j},
\end{equation}
for arbitrary $\mathbf{f} \in \mathcal{N}(\mathbf{G})$, where $c_0=1$, $c_{m+1} = -1$, and 
$$
f_0(x_0) := e^{-\pi \langle x_0 , \mathcal{Q}_+ x_0 \rangle},\;\;\;
f_{m+1}(x_{m+1}) := e^{-\pi \langle x_{m+1} , \mathcal{Q}_- x_{m+1} \rangle}.  
$$
Our strategy is to use a heat-flow monotonicity argument and to set things up we regard the left-hand side of \eqref{e:flow_goal} as $Q(1)$, where
\begin{equation}\label{e:Qdefn}
\mathfrak{Q}(t): = 
 \int_{\mathbb{R}^n} U(t,x) \, dx
\end{equation}
and
\[
U(t,x) := t^{-\alpha} \prod_{j=0}^{m+1} u_j(t, B_jx)^{c_j} 
\]
with
\[
\alpha := \frac{1}{2}\bigg(n- \sum_{j=0}^{m+1} c_j n_j \bigg).
\] 
Here, the function $u_j$ satisfies the heat equation
\begin{equation} \label{e:ujequation}
\partial_t {u}_j = \frac1{4\pi} {\rm div} \, (A_j^{-1}\nabla {u}_j ), \qquad {u}_j(1,x_j) = f_j(x_j)
\end{equation} 
for $j = 0,\ldots,m+1$, where 
$$
A_0:= \mathcal{Q}_+,\;\;\; A_{m+1}:=\mathcal{Q}_-. 
$$
In order to prove \eqref{e:flow_goal}, it suffices to show that $\mathfrak{Q}$ given by \eqref{e:Qdefn} is non-increasing on  $(1,\infty)$. Indeed, from the non-increasingness of $\mathfrak{Q}$ we may obtain
	\begin{align*}
	\int_{\mathbb{R}^n } \prod_{j=0}^{m+1} f_j(B_jx)^{c_j} \, dx = Q(1) \geq \liminf_{t\to \infty} \mathfrak{Q}(t).
	\end{align*}
	Furthermore, by an elementary change of variables we may write
	$$
	\mathfrak{Q}(t)
	=
	\int_{\mathbb{R}^n} \prod_{j=0}^{m+1} \big( t^{n_j/2}u_j(t, \sqrt{t} B_jy)\big)^{c_j} \, dy
	$$
	and thus Fatou's lemma and \eqref{e:largetime} imply
	\begin{align*}
	\liminf_{t\to\infty} \mathfrak{Q}(t)
	&\ge
	\int_{\mathbb{R}^n} \prod_{j=0}^{m+1} \bigg( ( {\rm det}\, A_j )^{1/2}  e^{-\pi \langle B_jy , A_j B_jy \rangle} \int_{\mathbb{R}^{n_j}} f_j \bigg)^{c_j} \, dy \\
	&=
	\bigg( \frac{ \prod_{j=1}^{m}( {\rm det}\, A_j )^{c_j} }{{\rm det}\,  \widetilde{M}}\bigg)^{1/2} \prod_{j=1}^{m} \bigg( \int_{\mathbb{R}^{n_j}} f_j\bigg)^{c_j}.
	\end{align*}
	
	In order to verify that $\mathfrak{Q}$ is non-increasing we shall make use of Theorem \ref{t:cp_IBL}. Thanks to the assumption \textit{(2)} we know that \eqref{e:closurekey1} and \eqref{e:closurekey2} are satisfied for $j=1,\ldots,m$. Hence we deduce from Theorem \ref{t:cp_IBL} that $U$ satisfies
\[
\partial_t U \leq \frac{1}{4\pi} {\rm div} (\widetilde{M}^{-1} \nabla U)
\]
and we would like to rigorously argue along the lines of \eqref{e:mono_formal} in order to show that $\mathfrak{Q}$ is non-increasing. First, we justify the integration by parts step, and then show that one can interchange the time derivative and the integral. 

Take $\varepsilon > 0$ sufficiently small (specified momentarily) and note the bound 
\[
\prod_{j=1}^m u_j(t,B_jx)^{c_j}  \leq C_*(t,\varepsilon) (1 + |x|^2)^N e^{-(1-\varepsilon) \pi \langle  x,t^{-1}M_+x  \rangle} 
\]
given by  \eqref{e:Ubound} in Lemma \ref{l:heatdecay}, where $N := \sum_{j=m_++1}^m |c_j|n_j$, $M_+ := M(\mathbf{B}_+,\mathbf{c}_+;\mathbf{A}_+)$, and $C_*(t,\varepsilon)$ is a constant depending on $t$, $\mathbf{B}$, $\mathbf{c}$, $\mathbf{A}$, $\mathbf{G}$, and $\mathbf{f}$, which is locally uniformly bounded in $t$. Since
\[
u_{j}(t,x_{j}) = t^{-\frac{n_{j}}{2}} e^{-\pi \langle x_{j}, t^{-1} A_{j} x_{j} \rangle}
\]
for $j=0,m+1$ where we recall that $A_0=\mathcal{Q}_+$, $A_{m+1} = \mathcal{Q}_-$, $c_0=1$, and $c_{m+1} = -1$, we have
\begin{equation} \label{e:Ufinal}
U(t,x) \leq C_*(t,\varepsilon) (1 + |x|^2)^N e^{-\pi \langle x, t^{-1} P(\varepsilon)x \rangle}
\end{equation}
with
\[
P(\varepsilon) := (1-\varepsilon)M_+ + \mathcal{Q}. 
\]
Since $\mathbf{A} \in \Lambda(\mathbf{B},\mathbf{c},\mathcal{Q})$ it follows that $M_+ + \mathcal{Q}> 0$ and thus $P(\varepsilon) > 0$ by choosing $\varepsilon > 0$ sufficiently small depending on $\mathbf{B}, \mathbf{c}$ and $\mathbf{A}$; consequently, $U$ is rapidly decreasing in space locally uniformly in time. By \eqref{e:Uspace} and \eqref{e:vbound} we have 
\[
|\nabla U(t,x)| \leq C_*(t,\varepsilon) U(t,x) \sum_{j=0}^{m+1} |c_j| (1 + |x|)^{n_j}
\]
and therefore $|\nabla U(t,x)|$ is rapidly decreasing in space locally uniformly in time. Hence, from the divergence theorem we have
\[
\lim_{R \to \infty} \int_{|x| \leq R} {\rm div} (\widetilde{M}^{-1}\nabla U)(t,x) \, dx = 0
\]
for each fixed $t > 1$.

In order to see that $\mathfrak{Q}'(t) = \int \partial_t U(t,x) \, dx$, we use the identity \eqref{e:Utime} along with the bounds \eqref{e:timebound} and \eqref{e:Ufinal} to see that $|\partial_t U(t,x)|$ is rapidly decreasing in space locally uniformly in time.

From the above, we have 
\[
\mathfrak{Q}'(t) = \int_{\mathbb{R}^n} \partial_t U(t,x) \, dx \leq \frac{1}{4\pi} \int_{\mathbb{R}^n} {\rm div} (\widetilde{M}^{-1} \nabla U)(t,x) \, dx = 0
\]
and we have the desired monotonicity of $\mathfrak{Q}$ on $(1,\infty)$.
\end{proof}

\begin{remarks}
(1) One may establish Theorem \ref{t:GkerGExt} in a similar manner to our sketch proof of Theorem \ref{t:Geo} in Section \ref{section:prelims}. The same line of reasoning was used in \cite[Proposition 8.9]{BCCT} in the case of the forward Brascamp--Lieb inequality and their heat-flow argument was abstracted in  \cite[Lemma 2.6]{BCCT}.  By proceeding via Theorem \ref{t:cp_IBL} we have kept our proof self contained, and, as explained in the previous section, we believe that the closure property in Theorem \ref{t:cp_IBL} is of wider independent interest.

(2) An inspection of the arguments in Section \ref{section:B_0=0_heatflow} reveals that the full strength of the non-degeneracy assumption (i.e. both \eqref{e:NondegBW} and \eqref{e:NondegG}) was not required, and the condition \eqref{e:NondegG} can be dropped at this stage. The condition \eqref{e:NondegG} will, however, be important for the arguments in the forthcoming Section \ref{section:B_0=0_globaldata}.

(3) One can make use of the above argument with Theorem \ref{t:Forwardcp} instead of Theorem \ref{t:cp_IBL} to derive the analogous statement to Theorem \ref{t:GkerGExt} for the forward Brascamp--Lieb inequality as follows. Let $\mathcal{Q}\ge0$, $\mathbf{G} = (G_j)_{j = 1}^m$, with $G_j \in S_+(\mathbb{R}^{n_j})$, and the Brascamp--Lieb datum $(\mathbf{B},\mathbf{c})$ be non-degenerate in the sense of \cite{BCCT} (namely $c_j>0$, $B_j$ is surjective for all $j=1,\ldots,m$, and $\bigcap_{j=1}^m {\rm ker}\, B_j = \{0\}$). Then we have  
$$
{\rm F}(\mathbf{B},\mathbf{c},\mathcal{Q},\mathbf{G})  = {\rm BL}(\mathbf{B},\mathbf{c},\mathcal{Q};\mathbf{A})
$$
for some $\mathbf{A} \leq \mathbf{G}$ if and only if $\mathbf{A}$ satisfies 
$$
A_j^{-1} - B_j\widetilde{M}(\mathbf{A})B_j^*\ge0,\;\;\; (A_j^{-1} - B_j\widetilde{M}(\mathbf{A})B_j^*)(G_j-A_j)=0. 
$$
\end{remarks}

\subsection{Approximation by amplifying data and the proof of Theorem \ref{t:QMain}} \label{section:B_0=0_globaldata}

In the case of the forward Brascamp--Lieb inequality, in order to reduce to the gaussian-extremizable case, the argument in \cite{BCCT} naturally made use of so-called \emph{localized} data, which means $B_j = \mathrm{id}$ and $c_j = 1$ for some $j \in \{1,\ldots,m\}$. In the framework of the inverse inequality, and in particular when $\mathbf{B}_+$ is bijective, the condition for localized data corresponds to the case $m_+ = 1$ and such a restrictive class of data cannot be expected to play an important role in reducing to the gaussian-extremizable case. Instead, we introduce the following notion.
\begin{definition}[Amplifying data] \label{d:globalised_data}
	The generalized Brascamp--Lieb datum $(\mathbf{B},\mathbf{c}, \mathcal{Q}, \mathbf{G})$ is said to \emph{amplifying} if  
	\begin{equation*}
	B_j = {\rm id}_{\mathbb{R}^{n}} \quad \text{and} \quad |c_j| > \max\{c_1,\ldots,c_{m_+}\} - 1
	\end{equation*}
	hold for some $j\in \{m_++1,\ldots,m\}$.
\end{definition}
The crucial properties of amplifying data is that they are gaussian extremizable and are able to approximate (in an appropriate sense) any non-degenerate data; we establish these facts in Lemmas \ref{l:GEQNega} and \ref{l:approx} below.
\begin{lemma}\label{l:GEQNega}
	Suppose the non-degenerate generalized Brascamp--Lieb datum $(\mathbf{B}, \mathbf{c}, \mathcal{Q}, \mathbf{G})$ is amplifying. Then $(\mathbf{B}, \mathbf{c}, \mathcal{Q}, \mathbf{G})$ is gaussian extremizable and
	in particular, from Theorem \ref{t:GkerGExt}, we have
	$$
	{\rm I}(\mathbf{B}, \mathbf{c},\mathcal{Q},\mathbf{G})
	=
	{\rm I}_{\mathbf{g}}(\mathbf{B}, \mathbf{c},\mathcal{Q},\mathbf{G}).
	$$
\end{lemma}

\begin{proof}
	Without loss of generality, we may assume 
	$$
	B_m = {\rm id}_{\mathbb{R}^n}  \quad \text{and} \quad |c_m| > \max\{c_1,\ldots,c_{m_+}\}-1. 
	$$
	In this proof, we suppress the dependence on $(\mathbf{B},\mathbf{c})$ and write
	\[
	M(\mathbf{A}) = \sum_{j=1}^m c_j B_j^*A_j B_j.
	\]
	
	Our first simple but important remark is that 
	\begin{equation}\label{e:FiniteQ}
	{\rm I}_{\mathbf{g}}(\mathbf{B}, \mathbf{c}, \mathbf{G}, \mathcal{Q}) < \infty.
	\end{equation}
	To see this, first note that $B_{0}^*\mathcal{Q}_+B_{0}  + M(\mathbf{G}_+) - B_{m+1}^*\mathcal{Q}_-B_{m+1} > 0$ follows immediately from \eqref{e:NondegG}.
	Thus, if we consider $\mathbf{A} \leq \mathbf{G}$ such that 
	$
	\mathbf{A}_+ = \mathbf{G}_+
	$
	and the components of $\mathbf{A}_-$ are chosen to be sufficiently small depending on $(\mathbf{B}, \mathbf{c}, \mathcal{Q}, \mathbf{G})$, then
	\[
	M(\mathbf{A}) + \mathcal{Q} = B_{0}^*\mathcal{Q}_+B_{0} +M(\mathbf{A}_+) - B_{m+1}^*\mathcal{Q}_-B_{m+1} + M(\mathbf{A}_-) > 0.
	\]
	Hence we clearly have \eqref{e:FiniteQ}.

	Next, by a continuous extension 
	\[
	{\rm I}_{\mathrm{g}}(\mathbf{B}, \mathbf{c}, \mathcal{Q},  \mathbf{G}) = \inf_{\substack{0 \leq A_j \leq G_j \\ j=1,\ldots,m}} {\rm BL}(\mathbf{B}, \mathbf{c}, \mathcal{Q}; \mathbf{A})
	\] 
	and thus there exists $\mathbf{A}^{\star}$ such that $0 \leq A_j^\star \leq G_j$ for $j=1,\ldots,m$ and
	\[
	{\rm I}_{\mathrm{g}}(\mathbf{B}, \mathbf{c}, \mathbf{G}, \mathcal{Q}) = {\rm BL}(\mathbf{B}, \mathbf{c}, \mathcal{Q}; \mathbf{A}^\star).
	\] 
	Our proof will be complete once we show that $A_j^{\star} > 0$ for each $j=1,\ldots,m$ and $\mathbf{A}^{\star} \in \Lambda(\mathbf{B},\mathbf{c},\mathcal{Q})$. The latter claim is easily dealt with since $ \mathbf{A}^\star  \notin \Lambda( \mathbf{B}, \mathbf{c}, \mathcal{Q})$ means ${\rm BL}(\mathbf{B}, \mathbf{c}, \mathcal{Q}; \mathbf{A}^\star) = \infty$ and this contradicts \eqref{e:FiniteQ}.
	
	Our remaining goal is to show $A_j^{\star} > 0$ for each $j=1,\ldots,m$. We suppose, for a contradiction, that $\det A_\ell^\star = 0$ for some $\ell \in \{1,\ldots,m\}$ and consider any $\mathbf{A}^{(\varepsilon)} \leq \mathbf{G}$ which converges to $\mathbf{A}^{\star}$. We shall show that
	\begin{equation} \label{e:blowup}
	\lim_{\varepsilon \to 0} {\rm BL}(\mathbf{B}, \mathbf{c}, \mathcal{Q}; \mathbf{A}^{(\varepsilon)}) = \infty
	\end{equation}
	and thus contradict \eqref{e:FiniteQ}. In order to show \eqref{e:blowup}, it clearly suffices to consider those $\varepsilon > 0$ such that $M(\mathbf{A}^{(\varepsilon)}) + \mathcal{Q} > 0$ in which case we have
	\begin{equation} \label{e:BLeps}
	{\rm BL}(\mathbf{B}, \mathbf{c}, \mathcal{Q}; \mathbf{A}^{(\varepsilon)}) = \frac{\prod_{j=1}^{m_+} ({\rm det}\, {A}_j^{(\varepsilon)})^{c_j}  \prod_{k=m_++1}^m ({\rm det}\, {A}_k^{(\varepsilon)})^{- |c_k|} }{{\rm det}\, (M(\mathbf{A}^{(\varepsilon)}) + \mathcal{Q})}.
	\end{equation}
	Suppose first $\det A_\ell^\star = 0$ with $\ell \in \{1,\ldots,m_+\}$ and, without loss of generality, we suppose $\ell=1$. In this case we estimate from below 
	\begin{align*}
	{\rm BL}(\mathbf{B}, \mathbf{c}, \mathcal{Q}; \mathbf{A}^{(\varepsilon)}) \geq C(\mathbf{G}) \frac{   ({\rm det}\, A_{m}^{(\varepsilon)})^{-|c_m|}} {{\rm det}\, (B_0^*\mathcal{Q}_+B_0 + M(\mathbf{A}_+^{(\varepsilon)}))} \prod_{j=1}^{m_+} ({\rm det}\, A_j^{(\varepsilon)})^{c_j},
	\end{align*}
	where $C(\mathbf{G})$ is a positive constant depending only on $\mathbf{G}$. Here we have also used the fact determinant respects the ordering of semi-definite positive matrices. Since $B_m = {\rm id}_{\mathbb{R}^n}$ we have
	\[
	B_0^*\mathcal{Q}_+B_0 + M(\mathbf{A}_+^{(\varepsilon)}) - |c_m|A_m^{(\varepsilon)} \geq M(\mathbf{A}^{(\varepsilon)}) + \mathcal{Q} > 0
	\]
	and therefore $\det (B_0^*\mathcal{Q}_+B_0 + M(\mathbf{A}_+^{(\varepsilon)})) \geq |c_m|^n \det A_m^{(\varepsilon)}$. Also, the bijectivity of $(B_0,\mathbf{B}_+)$ implies	
	$$
	{\rm det}\, (B_0^*\mathcal{Q}_+B_0 + M(\mathbf{A}_+^{(\varepsilon)}) )=  {\rm det}\, \mathcal{Q}_+ \prod_{j=1}^{m_+} c_j^{n_j} {\rm det}\, A_j^{(\varepsilon)}. 
	$$ 
	Therefore
	\[
	{\rm BL}(\mathbf{B}, \mathbf{c}, \mathcal{Q}; \mathbf{A}^{(\varepsilon)}) \geq C(\mathbf{c},\mathbf{G}) ({\rm det}\, \mathcal{Q}_+)^{-|c_m|-1}\prod_{j=1}^{m_+} ({\rm det}\, A_{j}^{(\varepsilon)} )^{c_j-1-|c_m|}
	\]
	and hence, using that $c_j - 1 - |c_m| < 0$ for $j = 1,\ldots,m_+$, we have
	\[
	{\rm BL}(\mathbf{B}, \mathbf{c}, \mathcal{Q}; \mathbf{A}^{(\varepsilon)}) \geq C(\mathbf{c},\mathbf{G},\mathcal{Q}_+)  ({\rm det}\, A_{1}^{(\varepsilon)} )^{c_1-1-|c_m|}
	\]
	for appropriate positive constants $C(\mathbf{c},\mathbf{G})$ and $C(\mathbf{c},\mathbf{G},\mathcal{Q}_+)$. Since $\det A_1^{(\varepsilon)} \to 0$ and $c_1 - 1 - |c_m| < 0$ we obtain \eqref{e:blowup}.

	In the remaining case we have $\det A_\ell^\star = 0$ with $\ell \in \{m_+ + 1,\ldots,m\}$ and we may suppose $\det A_j^\star > 0$ for each $j \in \{1,\ldots,m_+\}$ (otherwise the above argument applies). In this case it is clear that
	$\prod_{j=1}^{m_+} ({\rm det}\, {A}_j^{(\varepsilon)})^{c_j} \geq C(\mathbf{A}^\star)$ for sufficiently small $\varepsilon > 0$ and appropriate positive constant $C(\mathbf{A}^\star)$. Thus it is clear that the numerator in \eqref{e:BLeps} tends to infinity as $\varepsilon$ tends to zero. Also, ${\rm det}\, (M(\mathbf{A}^{(\varepsilon)}) + \mathcal{Q}) \leq 2 {\rm det} \, (M(\mathbf{A}^\star) + \mathcal{Q}) < \infty$ for all $\varepsilon > 0$ sufficiently small, and hence \eqref{e:blowup} trivially follows.
\end{proof}

The following lemma ensures the approximation of arbitrary generalized Brascamp--Lieb datum by amplifying data. 
\begin{lemma}\label{l:approx}
	Suppose $(\mathbf{B}, \mathbf{c}, \mathcal{Q}, \mathbf{G})$ is a non-degenerate generalized Brascamp--Lieb datum. Then, for any $c_+>0$ we have 
	\begin{equation}\label{e:ApproxQNega}
	{\rm I}(\mathbf{B}, \mathbf{c}, \mathcal{Q}, \mathbf{G})
	=
	\lim_{\lambda\downarrow 0} \lambda^{nc_+/2}
	{\rm I}((\mathbf{B},{\rm id}_{\mathbb{R}^n}), (\mathbf{c}, -c_+),\mathcal{Q},(\mathbf{G},\lambda {\rm id}_{\mathbb{R}^n}))
	\end{equation}
	and 
	\begin{equation}\label{e:ApproxQNegaG}
	{\rm I}_{\mathrm{g}}(\mathbf{B}, \mathbf{c}, \mathcal{Q}, \mathbf{G})
	=
	\lim_{\lambda\downarrow 0} \lambda^{nc_+/2}
	{\rm I}_{\mathrm{g}}((\mathbf{B},{\rm id}_{\mathbb{R}^n}), (\mathbf{c}, -c_+), \mathcal{Q}, (\mathbf{G},\lambda {\rm id}_{\mathbb{R}^n})).
	\end{equation}
\end{lemma}

\begin{proof}
	For arbitrary $\mathbf{f} \in \mathcal{N}(\mathbf{G})$, since $(B_0,\mathbf{B}_+)$ is bijective, we may apply the dominated convergence theorem to see 
	\begin{align*}
	&\int_{\mathbb{R}^n} e^{-\pi\langle x,\mathcal{Q}x\rangle} \prod_{j=1}^m f_j( B_jx)^{c_j} \, dx\\
	=& 
	\lim_{\lambda\downarrow 0} \int_{\mathbb{R}^n} e^{-\pi\langle x,\mathcal{Q}x\rangle} \prod_{j=1}^m f_j (B_jx)^{c_j}  e^{\pi c_+ \langle x, \lambda {\rm id}_{\mathbb{R}^n} x \rangle }  \, dx  \\
	=& 
	\lim_{\lambda\downarrow 0} \lambda^{nc_+/2}  \int_{\mathbb{R}^n} e^{-\pi\langle x,\mathcal{Q}x\rangle} \prod_{j=1}^m f_j (B_jx)^{c_j}  g_{\lambda {\rm id}_{\mathbb{R}^n} }(x)^{- c_+ } \, dx  \\
	\ge &
	\lim_{\lambda\downarrow 0} \lambda^{nc_+/2} 
	{\rm I}((\mathbf{B}, {\rm id}_{\mathbb{R}^n}), (\mathbf{c}, -c_+),\mathcal{Q},(\mathbf{G}, \lambda {\rm id}_{\mathbb{R}^n})) \prod_{j=1}^m \bigg( \int_{\mathbb{R}^{n_j}} f_j \bigg)^{c_j} 
	\end{align*}
	which, thanks to Lemma \ref{l:ApproxG}, shows 
	\begin{equation*}\label{e:28Mar1}
	{\rm I}(\mathbf{B}, \mathbf{c},\mathcal{Q},\mathbf{G}) 
	\ge 
	\lim_{\lambda\downarrow 0} \lambda^{nc_+/2} 
	{\rm I}((\mathbf{B}, {\rm id}_{\mathbb{R}^n}), (\mathbf{c}, -c_+),\mathcal{Q}, (\mathbf{G}, \lambda {\rm id}_{\mathbb{R}^n})). 
	\end{equation*}
	In order to obtain the converse inequality, for any $\lambda>0$ and any $(\mathbf{f}, f_{m+1}) \in \mathcal{T}(\mathbf{G}) \times \mathcal{T}(\lambda {\rm id}_{\mathbb{R}^n})$, we need to bound 
	\begin{align*}
	\int_{\mathbb{R}^n} e^{-\pi\langle x,\mathcal{Q}x\rangle} \prod_{j=1}^m f_j(B_jx)^{c_j} f_{m+1}(x)^{-c_+} \, dx   
	\end{align*}
	from below.  Since $f_{m+1} \in \mathcal{T}(\lambda {\rm id}_{\mathbb{R}^n})$, we can write $f_{m+1} = g_{\lambda {\rm id}_{\mathbb{R}^n}} \ast d\mu_{m+1}$ for some positive and finite Borel measure $d\mu_{m+1}$ and so we have   
	$$
	f_{m+1}(x) 
	= \lambda^{n/2} \int_{\mathbb{R}^n} e^{-\lambda|x-y|^2} \, d\mu_{m+1}(y) \le  \lambda^{n/2} \int_{\mathbb{R}^n} f_{m+1}
	$$
	uniformly in $x$. This yields 
	\begin{align*}
	& \int_{\mathbb{R}^n} e^{-\pi\langle x,\mathcal{Q} x\rangle} \prod_{j=1}^m f_j(B_jx)^{c_j} f_{m+1}(x)^{-c_+}\, dx \\
	& \qquad \qquad \ge \lambda^{-nc_+/2} {\rm I}(\mathbf{B}, \mathbf{c},\mathcal{Q},\mathbf{G})  \prod_{j=1}^m \bigg( \int_{\mathbb{R}^{n_j}} f_j \bigg)^{c_j} \bigg( \int_{\mathbb{R}^{n}} f_{m+1} \bigg)^{-c_+}
	\end{align*}
	which shows 
	\begin{equation*}\label{e:28Mar2}
	{\rm I}((\mathbf{B}, {\rm id}_{\mathbb{R}^n}), (\mathbf{c}, -c_+), \mathcal{Q}, (\mathbf{G}, \lambda {\rm id}_{\mathbb{R}^n})) \ge
	\lambda^{- nc_+/2} {\rm I}(\mathbf{B}, \mathbf{c},\mathcal{Q},\mathbf{G}),
	\end{equation*}
	and we conclude \eqref{e:ApproxQNega}. A similar argument yields \eqref{e:ApproxQNegaG}. 
\end{proof}

We are now in a position to remove the gaussian extremizability assumption in Theorem \ref{t:GkerGExt} and thus establish Theorem \ref{t:QMain}. 
\begin{proof}[Proof of Theorem \ref{t:QMain}] 
	Choose any $c_+ >0$ such that $c_+> \max{\{c_1,\ldots,c_{m_+}\}} -1$. Then, for any $\lambda > 0$ it is clear that the augmented data $((\mathbf{B},{\rm id}_{\mathbb{R}^n}), (\mathbf{c}, -c_+),\mathcal{Q},(\mathbf{G},\lambda {\rm id}_{\mathbb{R}^n}))$ is amplifying. Also, since $(\mathbf{B},{\rm id}_{\mathbb{R}^n})_+ = \mathbf{B}_+$ the augmented data  is non-degenerate and hence we may apply Lemma \ref{l:GEQNega} to give
	$$
	{\rm I} ((\mathbf{B},{\rm id}_{\mathbb{R}^n}), (\mathbf{c}, -c_+), \mathcal{Q},(\mathbf{G}, \lambda{\rm id}_{\mathbb{R}^n}))
	=
	{\rm I}_{\mathrm{g}} ((\mathbf{B},{\rm id}_{\mathbb{R}^n}), (\mathbf{c}, -c_+), \mathcal{Q}, (\mathbf{G}, \lambda{\rm id}_{\mathbb{R}^n})).
	$$
	Multiplying both sides by $\lambda^{nc_+/2}$ and taking the limit $\lambda \downarrow 0$, we obtain the desired conclusion from Lemma \ref{l:approx}.
\end{proof}

\section{Further applications and remarks}\label{section:furtherapps}
 
\subsection{Regularized forms of the Young convolution inequality} \label{Sub:Y}
For $c \in \mathbb{R}$, we introduce the constant 
\[
A_c := \bigg( \frac{|1-c|^{1-c}}{|c|^c} \bigg)^{1/2}.
\]
The sharp form of the forward and inverse Young convolution inequality on $\mathbb{R}$ may be expressed as
\begin{equation} \label{e:YdualF}
{\rm F}(\mathbf{B},\mathbf{c}) = A_{c_0}A_{c_1}A_{c_2} \qquad (c_0,c_1,c_2 \in (0,1])
\end{equation}
and
\begin{equation} \label{e:YdualI}
{\rm I}(\mathbf{B},\mathbf{c}) = A_{c_0}A_{c_1}A_{c_2} \qquad \text{($c_j < 0$ for some $j$, and $c_k \in [1,\infty)$ for $k \neq j$)},
\end{equation}
where, in both cases, the $c_j$ satisfy the scaling condition $c_0 + c_1 + c_2 = 2$, and $B_j : \mathbb{R}^2 \to \mathbb{R}$ are given by
	\begin{equation}\label{e:Y_Bjdefn}
		B_0(x,y):=x,\;\;\; B_1(x,y):=y,\;\;\;B_2(x,y):=x-y. 
	\end{equation}
The forward version \eqref{e:YdualF} was established independently by Beckner \cite{Beck} and Brascamp--Lieb \cite{BL}. In the same paper, Brascamp--Lieb \cite{BL} established\footnote{Strictly speaking, it was proved in \cite{BL} that the inequality $\| f_1 * f_2 \|_{p_0} \geq A_{c_0}A_{c_1}A_{c_2} \|f_1\|_{p_1}\|f_2\|_{p_2}$ holds when $p_0,p_1,p_2 \in (0,1]$, where $c_0:=1-\frac1{p_0}, c_1:= \frac1{p_1}, c_2:= \frac1{p_2}$. As observed in \cite[Example 2.14]{BW}, in dual form \eqref{e:YdualI}, it becomes apparent that there is a symmetry amongst $c_0,c_1,c_2$, or equivalently, $p_0',p_1,p_2$. As clarified in  \cite[Example 2.14]{BW}, the condition on the $c_j$ in \eqref{e:YdualI}  is necessary and corresponds to the bijectivity of $\mathbf{B}_+$.} \eqref{e:YdualI}.
 
Under the scaling condition $c_0 + c_1 + c_2 = 2$, one has an invariance of extremizers under the isotropic rescaling $f_j \to f_j(R \, \cdot)$ and so it follows that one cannot hope to improve the constants in \eqref{e:YdualF} and \eqref{e:YdualI} even if one only considers $f_j$ of type\footnote{In this discussion, we use the notation $\frac{1}{\sigma_j}$ to facilitate a comparison with the related results in \cite{BB}.} $\frac{1}{\sigma_j}$ for any $\sigma_j > 0$. However, by considering such $f_j$, one may relax the scaling condition and below we present a result of this type. To state the result, we use the notation
\[
\widetilde{A}_{c,\sigma} := \bigg(\frac{\sigma^{1-c}}{c} \bigg)^{1/2} \qquad (c,\sigma > 0).
\]
\begin{corollary}\label{p:DualYoung}
	Let $\mathbf{B}$ be given by \eqref{e:Y_Bjdefn}. Given $c_0 < 1$ and $c_1,c_2 > 0$, suppose $\sigma_0, \sigma_1,\sigma_2>0$ satisfy  
	\begin{equation}\label{e:DefExp}
	\frac{\sigma_0}{1-c_0} = \frac{\sigma_1}{c_1} + \frac{\sigma_2}{c_2}
	\end{equation}
	and set $\mathbf{G} = (\sigma_0^{-1},\sigma_1^{-1},\sigma_2^{-1})$.
 
(1) Suppose $c_0,c_1,c_2\in (0,1)$. Then 
		\begin{equation}\label{e:FBLYoung}
			{\rm F}(\mathbf{B},\mathbf{c},\mathbf{G})  = \frac{\widetilde{A}_{c_1,\sigma_1} \widetilde{A}_{c_2,\sigma_2}}{\widetilde{A}_{1-c_0,\sigma_0}}
		\end{equation}
		holds if and only if 
		\begin{equation}\label{e:CondRegFYoung}
				\frac{c_0(1-c_0)}{\sigma_0} \ge \max\bigg\{\frac{c_1(1-c_1)}{\sigma_1},\frac{c_2(1-c_2)}{\sigma_2}\bigg\}.
 			\end{equation}
		
(2) Suppose $c_0 < 0, c_1,c_2 \in [1,\infty)$. Then 
		\begin{equation}\label{e:IBLYoung}
			{\rm I}(\mathbf{B},\mathbf{c},\mathbf{G})  = \frac{\widetilde{A}_{c_1,\sigma_1} \widetilde{A}_{c_2,\sigma_2}}{\widetilde{A}_{1-c_0,\sigma_0}}
		\end{equation}
		holds if and only if 
		\begin{equation}\label{e:CondRegRYoung}
				\frac{c_0(1-c_0)}{\sigma_0} \le \min \bigg\{\frac{c_1(1-c_1)}{\sigma_1},\frac{c_2(1-c_2)}{\sigma_2}\bigg\}.
 		\end{equation} 
\end{corollary}

\begin{remarks}
(1) Although the relaxation of the scaling condition is not explicit in the above statement, one can show that, when $c_0,c_1,c_2\in (0,1)$, if \eqref{e:CondRegFYoung} holds for some $\sigma_j > 0$ satisfying \eqref{e:DefExp}, then $c_0 + c_1 + c_2 \geq 2$ holds. Similarly, when $c_0 < 0, c_1,c_2 \in [1,\infty)$, it can be shown that if \eqref{e:CondRegRYoung} holds for some $\sigma_j > 0$ satisfying \eqref{e:DefExp}, then $c_0 + c_1 + c_2 \leq 2$ holds.

(2) Sharp forms of Young convolution inequalities have been previously considered in \cite{BB} (not in the dual setting, but in terms of the norm inequality). In particular, it follows from \cite[Corollary 7]{BB} that \eqref{e:FBLYoung} holds under the condition \eqref{e:CondRegFYoung}. However, the condition in \cite{BB} is that
\begin{equation*}
c_0 + c_1 + c_2 \geq 2
\end{equation*}
and, for some $\alpha_1,\alpha_2 \in [0,1]$,
\begin{equation*}
c_0 + \alpha_1 c_1 + \alpha_2 c_2 = 2 \qquad \text{and} \qquad c_1(1 - \alpha_1 c_1) \sigma_2 = c_2(1- \alpha_2 c_2) \sigma_1.
\end{equation*}
Although is not immediately obvious, one can show the equivalence of this condition with \eqref{e:CondRegFYoung}. The inverse inequality \eqref{e:IBLYoung} under the relaxed scaling condition was not explicitly stated in \cite{BB}, but the arguments there can be modified to show that \eqref{e:IBLYoung} follows from \eqref{e:CondRegRYoung}.

(3) The meaning of \eqref{e:FBLYoung} and \eqref{e:IBLYoung} is that the regularized constants are attained ``on the boundary". In other words, for example, \eqref{e:IBLYoung} is equivalent to
\begin{equation}\label{e:21Sep-1}
	{\rm I}(\mathbf{B},\mathbf{c},\mathbf{G}) 
	=
	{\rm BL}(\mathbf{B},\mathbf{c};\mathbf{G}).
	\end{equation}
As one would expect, the proof will proceed via Theorem \ref{t:GkerGExt} and, essentially, our contribution in Corollary \ref{p:DualYoung} is showing that the conditions for gaussian extremizability in Theorem \ref{t:GkerGExt} can be equivalently expressed in the simple form  \eqref{e:CondRegRYoung} (and similarly for the forward inequality).
\end{remarks}

\begin{proof}[Proof of Corollary \ref{p:DualYoung}] 
We give the details for the inverse case (since the forward case is similar) and hence assume $c_0 < 0$, $c_1,c_2 \in [1,\infty)$. As remarked above, the goal is to show \eqref{e:21Sep-1} is equivalent to \eqref{e:CondRegRYoung}.

From $c_0 < 0$, $c_1,c_2 \in [1,\infty)$ it is immediate that the non-degenerate condition \eqref{e:DefNondeg} is satisfied for our datum and hence we may apply Theorem \ref{t:GkerGExt} to see that \eqref{e:21Sep-1} holds if and only if 
	\begin{equation}\label{e:GaussAtain}
		\Gamma_0(\mathbf{G}) \geq 0,\; \Gamma_j(\mathbf{G}) \leq 0\quad (j=1,2),
	\end{equation}
	where 
	\begin{equation}\label{e:Gamma_j}
	\Gamma_j(a_0,a_1,a_2):= a_j^{-1} - \frac{\sum_{k=0,1,2:k\neq j} c_ka_k}{c_0c_1a_0a_1 + c_0c_2a_0a_2 + c_1c_2a_1a_2 }
	\end{equation}
	for $a_0,a_1,a_2 > 0$. Indeed, a straightforward computation reveals that the condition \eqref{e:GaussAtain} coincides with \eqref{e:closurekey1} (with $\mathbf{A} = \mathbf{G}$) for our datum. 
	
	From \eqref{e:DefExp} it is easy to check that $\Gamma_0(\mathbf{G}) = 0$ and thus it suffices to show that 
	\begin{equation} \label{e:YconditionBN}
	\Gamma_j(\mathbf{G}) \leq 0 \quad \Leftrightarrow \quad 	
				\frac{c_0(1-c_0)}{\sigma_0} \le  \frac{c_j(1-c_j)}{\sigma_j}
	\end{equation} 
	for $j=1,2$. To see this for $j=1$, first note that \eqref{e:DefExp} yields
	$$
	\frac{c_0c_1}{\sigma_0\sigma_1} + \frac{c_0c_2}{\sigma_0\sigma_2} + \frac{c_1c_2}{\sigma_1\sigma_2} = \frac{1}{1-c_0} \frac{c_1c_2}{\sigma_1\sigma_2}
	$$
	and hence $\Gamma_1(\mathbf{G}) \leq 0$ is equivalent to
$$
	\sigma_1 \le (1 - c_0) \frac{\sigma_1\sigma_2}{ c_1c_2 } ( \frac{c_0}{\sigma_0} + \frac{c_2}{\sigma_2} ).
	$$ 
This, however, can be rearranged to $\frac{c_0(1-c_0)}{\sigma_0} \le  \frac{c_1(1-c_1)}{\sigma_1}$ by making use of \eqref{e:DefExp}. The equivalence \eqref{e:YconditionBN} for $j=2$ can be verified in the same manner and this completes our proof.
	\end{proof}

\subsection{Regularized Pr\'{e}kopa--Leindler inequality} \label{Sub:PL}
We begin by recalling the one dimensional Pr\'{e}kopa--Leindler inequality which states that if $c_1,c_2$ satisfy the scaling condition 
\begin{equation}\label{e:ScalePL}
	c_1+c_2=1,
\end{equation}
then for all nonnegative $f_1,f_2 \in L^1(\mathbb{R})$, 
\begin{equation}\label{e:PL}
 \bigg( \int_{\mathbb{R}} f_1 \bigg)^{c_1} \bigg( \int_{\mathbb{R}} f_2 \bigg)^{c_2} \leq \int_{\mathbb{R}} \esssup_{\substack{ x_1,x_2\in \mathbb{R}:\\ x= c_1x_1+c_2x_2}} f_1(x_1)^{c_1}f_2(x_2)^{c_2}\, dx. 
\end{equation}

The necessity of the scaling condition \eqref{e:ScalePL} is an easy consequence of the scaling argument.  
As in the case for the Young convolution inequality, one might expect to salvage the Pr\'{e}kopa--Leindler inequality in a scale-free case $c_1+c_2\neq1$ by restricting inputs to regularized datum.
For $c_1,c_2\in (0,1]$ and $\sigma_1,\sigma_2>0$, we define ${\rm PL}(\mathbf{c},(\sigma_1^{-1},\sigma_2^{-1}))\in [0,\infty]$ to be the sharp constant for the inequality 
\begin{equation}\label{e:NormalPL}
 \bigg( \int_{\mathbb{R}} f_1 \bigg)^{c_1} \bigg( \int_{\mathbb{R}} f_2 \bigg)^{c_2}  \le C \int_{\mathbb{R}} \esssup_{\substack{ x_1,x_2\in \mathbb{R}:\\ x= c_1x_1+c_2x_2}} f_1(x_1)^{c_1}f_2(x_2)^{c_2}\, dx,  
\end{equation}
where $(f_1,f_2)\in \mathcal{T}(\sigma_1^{-1})\times \mathcal{T}(\sigma_2^{-1})$.   
From Theorem \ref{t:RegFRBL}, then we have that, regardless of $c_1,c_2\in(0,1]$, 
\begin{equation} \label{e:PLinf}
{\rm PL}(\mathbf{c},(\sigma_1^{-1},\sigma_2^{-1})) = \bigg( \inf_{\substack{ 0<a_j \le \sigma_j^{-1}:\\ j=1,2}} \Phi_{\mathbf{c}} (a_1,a_2)\bigg)^{-1/2},
\end{equation}
where
\[
\Phi_{\mathbf{c}} (a_1,a_2):= a_1^{c_1}a_2^{c_2} \left(\frac{c_1}{a_1} + \frac{c_2}{a_2}\right).
\]
Moreover one can show that ${\rm PL}(\mathbf{c},(\sigma_1^{-1},\sigma_2^{-1})) < \infty$ as long as $c_1+c_2<1$ and $\sigma_1,\sigma_2 >0$. 
Here we use \eqref{e:PLinf} to identify the exact constant under certain conditions on $\sigma_1,\sigma_2$. 

\begin{corollary}\label{Cor:RegPreLei-1d}
	Let $c_1,c_2\in (0,1)$ satisfy $c_1+c_2<1$. 
	For $\sigma_1,\sigma_2>0$, it holds that  
	\begin{equation}\label{e:ConsRegPL}
		{\rm PL}(\mathbf{c},(\sigma_1^{-1},\sigma_2^{-1})) = \Phi_{\mathbf{c}}(\sigma_1^{-1},\sigma_2^{-1})^{-1/2}=\left( \frac{\sigma_1^{c_1}\sigma_2^{c_2}}{c_1\sigma_1+c_2\sigma_2} \right)^{1/2}
	\end{equation}
	if and only if 
	\begin{equation}\label{e:CondRegPL}
		c_1\sigma_1 + c_2\sigma_2 \le \min\{\sigma_1,\sigma_2\}.
	\end{equation}
\end{corollary}

\begin{proof}
		We begin with the sufficiency part. Observe that 
	$$
	\partial_1 \Phi_{\mathbf{c}}(a_1,a_2) = c_1a_1^{c_1-1}a_2^{c_2} \bigg( -\frac{1-c_1}{a_1} + \frac{c_2}{a_2} \bigg),\;\;\;
	\partial_2 \Phi_{\mathbf{c}}(a_1,a_2) = c_2a_1^{c_1}a_2^{c_2-1} \bigg( \frac{c_1}{a_1} -\frac{1-c_2}{a_2}  \bigg)
	$$
	and hence 
	$$
	\partial_1\Phi_{\mathbf{c}}(a_1,a_2) = \partial_2\Phi_{\mathbf{c}}(a_1,a_2) = 0 \quad \Leftrightarrow \quad c_1+c_2=1.
	$$
	In particular, in the case $c_1+c_2<1$, there is no extremum on $\{ (a_1,a_2): 0<a_j<\sigma_j^{-1} \}$. Hence the minimum of $\Phi_{\mathbf{c}}(a_1,a_2)$ is attained on the boundary of $[0,\sigma_1^{-1}]\times [0,\sigma_2^{-1}]$. Furthermore, since $\Phi_{\mathbf{c}}(a_1,a_2) = \infty$ if either $a_1=0$ or $a_2=0$, we see that 
	$$
	\inf_{0<a_j<\sigma_j^{-1}} \Phi_{\mathbf{c}} (a_1,a_2) 
	=
	\min{ \bigg( \inf_{0<a_1<\sigma_1^{-1}} \Phi_{\mathbf{c}}(a_1,\sigma_2^{-1}), \inf_{0<a_2<\sigma_2^{-1}} \Phi_{\mathbf{c}}(\sigma_1^{-1},a_2) \bigg) }. 
	$$ 
	Hence it suffices to investigate $\Phi_{\mathbf{c}}(a_1,\sigma_2^{-1}) $ and $\Phi_{\mathbf{c}}(a_1,\sigma_2^{-1}) $. For fixed $\sigma_1^{-1}, \sigma_2^{-1}>0$, we see from the formula of $\partial_j\Phi_{\mathbf{c}}$ that 
	$$
	\partial_1\Phi_{\mathbf{c}}(a_1,\sigma_2^{-1}) = 0 
	\quad \Leftrightarrow \quad a_1 =a_1^*:= \frac{1-c_1}{c_2\sigma_2}
	$$ 
	and similarly 
	$$
	\partial_2\Phi_{\mathbf{c}}(\sigma_1^{-1},a_2) = 0 
	\quad \Leftrightarrow \quad a_2 =a_2^*:= \frac{1-c_2}{c_1\sigma_1}. 
	$$ 
	Now we appeal to our assumption $c_1\sigma_1 + c_2\sigma_2 \le \min{ (\sigma_1,\sigma_2) } $. In fact, from this we see that 
	$$
	a_1^* \notin [0,\sigma_1^{-1}),\;\;\; a_2^* \notin [0, \sigma_2^{-1}). 
	$$
	Namely, $\Phi_{\mathbf{c}}(a_1,\sigma_2^{-1})$ is monotone on $[0,\sigma_1^{-1}]$ and similarly $\Phi_{\mathbf{c}}(\sigma_1^{-1},a_2)$ is monotone on $[0,\sigma_2^{-1}]$. We conclude that 
	$$
	\inf_{0<a_1<\sigma_1^{-1}} \Phi_{\mathbf{c}}(a_1,\sigma_2^{-1})
	= 
	\Phi_{\mathbf{c}}(\sigma_1^{-1},\sigma_2^{-1}),\;\;\; \inf_{0<a_2<\sigma_2^{-1}} \Phi_{\mathbf{c}}(\sigma_1^{-1},a_2)
	= 
	\Phi_{\mathbf{c}}(\sigma_1^{-1},\sigma_2^{-1}) 
	$$ 
	which implies \eqref{e:ConsRegPL}. 
	
	To show the necessity of \eqref{e:CondRegPL}, we define  
	$$
	\Psi(a_1,a_2):= \log\, \Phi_{\mathbf{c}}(a_1,a_2) = \sum_{j=1,2} c_j \log a_j + \log (c_1a_1^{-1}+c_2a_2^{-1}),\;\;\; a_1,a_2>0.
	$$
	Then from the assumption \eqref{e:ConsRegPL} we see that 
	$$
	\frac{d}{d\varepsilon}\Psi (\sigma_1^{-1} - \varepsilon, \sigma_2^{-1})|_{\varepsilon=0} := \lim_{\varepsilon \downarrow 0} \frac1{\varepsilon} \big( \Psi(\sigma_1^{-1} - \varepsilon, \sigma_2^{-1}) - \Psi(\sigma_1^{-1}, \sigma_2^{-1}) \big)\ge 0. 
	$$
	On the other hand, we have 
	$$
	\frac{d}{d\varepsilon}\Psi (\sigma_1^{-1} - \varepsilon, \sigma_2^{-1})|_{\varepsilon=0}
	= 
	-c_1\sigma_1 + \sigma_1  - \frac{c_2}{c_1 \sigma_2^{-1} + c_2 \sigma_1^{-1}}
	$$
	and hence it follows that 
	$$
	 \frac{c_2}{c_1 \sigma_2^{-1} + c_2 \sigma_1^{-1}} - (1-c_1) \sigma_1 \le 0
	$$
	which yields $c_1\sigma_1 + c_2\sigma_2 \le \sigma_1$. Similarly, by considering a perturbation in $a_2$, we can obtain $c_1\sigma_1 + c_2\sigma_2 \le \sigma_2$. 
\end{proof}

\subsection{Regularized forms of hypercontractivity inequalities} \label{Sub:H}
We conclude this section with some remarks about the forward and inverse hypercontractivity inequalities of the form
\begin{equation}\label{e:ForwardHC}
	\|e^{s\mathcal{L}}(F^{1/p}) \|_{L^{q}(d\gamma)} \le \bigg( \int_{\mathbb{R}} F\, d\gamma \bigg)^{1/p}  \qquad (p > 1, s > 0)
\end{equation}
for all nonnegative $F \in L^1(d\gamma)$, and 
\begin{equation}\label{e:ReersedHC}
	\|e^{s\mathcal{L}}(F^{1/p}) \|_{L^{q}(d\gamma)} \ge \bigg( \int_{\mathbb{R}} F\, d\gamma \bigg)^{1/p} \qquad (0 \neq p < 1, s > 0)
\end{equation}
for all positive $F \in L^1(d\gamma)$. Here, $(e^{s\mathcal{L}})_{s > 0}$ is the Ornstein--Uhlenbeck semigroup given by 
\begin{equation*}
	e^{s\mathcal{L}}F(x):= \int_{\mathbb{R}}F(e^{-s}x+(1-e^{-2s})^{1/2}y)\, d\gamma(y),
\end{equation*}
where $\gamma$ is the density function of the standard normal distribution
\[
d\gamma(x):= g_{1/(2\pi)}(x)dx.
\]
Also, the exponents $p,q \in \mathbb{R}$ and $s > 0$ in \eqref{e:ForwardHC} and \eqref{e:ReersedHC} satisfy the relation
\begin{equation}\label{e:Nelsontime}
	e^{2s} = \frac{q-1}{p-1},
\end{equation}
and the constant $1$ appearing in both inequalities is optimal.
The forward inequality is due to Nelson \cite{Nelson} and the inverse inequality is due to Borell \cite{Borell}; we refer the reader to, for example, \cite{BGL} for further details about the importance of inequalities of this kind.

There are a number of ways to obtain \eqref{e:ForwardHC} and \eqref{e:ReersedHC}, one of which is to write
\[
\|e^{s\mathcal{L}}(F^{1/p}) \|_{L^{q}(d\gamma)} = C \| (Fg_{1/(2\pi)})^{1/p} * g_\star \|_{L^q(dx)}
\]
where $C$ is a constant and $g_\star$ is an isotropic gaussian (both explicitly computable). From this expression, one may obtain \eqref{e:ForwardHC} and \eqref{e:ReersedHC} from the sharp form of the forward and inverse Young convolution inequalities; see \cite[Theorem 5]{Beck}. Since one of the inputs is a fixed gaussian, the scale-invariance property of the Young convolution inequality for general inputs ceases to hold, and one may thus expect to improve the constant in \eqref{e:ForwardHC} and \eqref{e:ReersedHC} by restricting to inputs $F$ which are regularized in an appropriate sense. One can obtain certain results of this type via Theorem \ref{t:QMain} (and its forward counterpart) by using the representation
 	\begin{align*}
		\|e^{s\mathcal{L}}(F^{1/p})\|_{L^{q}(d\gamma)} 
		= 
		C(p,s)
		\int_{\mathbb{R}^2} e^{-\pi\langle x, \mathcal{Q} x\rangle} \prod_{j=1,2} f_j(B_jx)^{c_j} \, dx,
	\end{align*}
	where 
	\begin{align*}
	&C(p,s):= (2\pi)^{\frac12(\frac1p + \frac1{q'})-1} (1-e^{-2s})^{-1/2}    ,\\
	&c_1:= \frac1p, \quad c_2 =\frac1{q'}, \qquad B_j(x_1,x_2):=x_j,
	\end{align*}
	and
	\begin{align*}
		\mathcal{Q}&:= \frac{1}{2\pi(1-e^{-2s})}
		\left(
		\begin{array}{cc}
			1-(1-e^{-2s})\frac1p & -e^{-s} \\
			-e^{-s} & 1-(1-e^{-2s})\frac1{q'}
		\end{array}
		\right),\\
		f_1(x_1)&:= F(x_1) g_{1/(2\pi)}(x_1),\;\;\; f_2(x_2):= \frac{e^{s\mathcal{L}}(F^{1/p})(x_2)^q}{\|e^{s\mathcal{L}}(F^{1/p})\|_{L^q(d\gamma)}^q} g_{1/(2\pi)}(x_2).
	\end{align*}
We refrain from explicitly stating such results here since stronger results have been obtained in collaboration with Hiroshi Tsuji in \cite{BNT} and we refer the reader there for precise statements. For instance, we also proved that the best constant for the logarithmic Sobolev inequality and Talagrand's inequality can be improved by restricting inputs to certain regularized functions. 

We also remark that it would be natural to consider bounds on more general gaussian kernel operators and to investigate the extent to which these are quantitatively improved upon by heat-flow regularization. For instance, it is already interesting to consider the gaussian kernel above for $p,q$ that do not satisfy \eqref{e:Nelsontime}. For such data, the non-degeneracy condition \eqref{e:NondegBW} fails to hold and moreover one cannot expect a nontrivial Brascamp--Lieb constant, see \cite{BW} and \cite{NT}. 
From similar reasoning, such data does not appear to fit into the framework of the forward-reverse Brascamp--Lieb inequality in Theorem \ref{t:RegFRBL}.
However, the second author and Tsuji \cite{NT,NT2} very recently observed that one can recover the Brascamp--Lieb inequality associated to such data if one restricts the inputs $f_1,f_2$ to be \textit{even} functions, and moreover showed how such an improvement implies the functional Blaschke--Santal\'{o} inequality\footnote{For further information regarding the functional Blaschke--Santal\'{o} inequality, we refer the reader to \cite{AKM,BallPhd, CGNT, CFM, Fathi, FraMeyMathZ, FMZSurvey, KW, LehecDirect, LehecYaoYao}. For more detailed discussion about this new link to the volume product, including Mahler's conjecture, we refer the interested reader to \cite{NT,NT2}. }.

\section*{Acknowledgements}
The first author would like to thank Jon Bennett for numerous enlightening conversations around the subject of this work. The second author would like to thank  the organizers of the Harmonic Analysis Seminar held at Shinshu University in 2020 where he learnt several ideas used in this work. Finally, both authors would like to express their sincere gratitude to an anonymous referee whose comments led to substantial improvements in the paper. In particular, the referee suggested to us that a result such as Theorem \ref{t:RegFRBL} should be true, and the manner in which it is equivalent to Theorem \ref{t:QMain}.

\section*{Appendix: On the equivalence between Theorems \ref{t:QMain} and \ref{t:RegFRBL}} \label{subsection:equivproof}

\begin{proof}[Theorem \ref{t:RegFRBL} $\Rightarrow$ Theorem \ref{t:QMain}]
Fix a non-degenerate datum $(\mathbf{B},\mathbf{c},\mathcal{Q},\mathbf{G})$ in the sense of \eqref{e:NondegBW} and  \eqref{e:NondegG}, where the $B_i$ are mappings from $\mathbb{R}^n$ to $\mathbb{R}^{n_i}$, $i = 1,\ldots, m$. To show Theorem \ref{t:QMain}, it suffices to show  
\begin{align}\label{e:RegIBL-Nov12}
	\int_{\mathbb{R}^n} e^{ -\pi \langle x, \mathcal{Q}x\rangle } \prod_{i=1}^{m_+} f_i(B_i x)^{c_i} \prod_{j=m_++1}^m h_j(B_jx)^{c_j}\, dx 
	\ge 
	{\rm I} \prod_{i=1}^{m_+} \bigg( \int_{\mathbb{R}^{n_i}} f_i \bigg)^{c_i}
\prod_{j=m_++1}^{m} \bigg( \int_{\mathbb{R}^{n_j}} h_j \bigg)^{c_j}\end{align}
for $f_i\in \mathcal{T}(G_i)$ and $h_j \in \mathcal{T}(G_j)$, where 
$$
{\rm I}:= \inf_{\substack{A_i\le G_i \\ i=1,\ldots,m} } \frac{ \prod_{i=1}^{m} \big( {\rm det}\, A_i \big)^\frac{c_i}2 }{  {\rm det}\, \big( \mathcal{Q} + \sum_{i=1}^{m} c_i B_i^* A_iB_i \big)^{\frac12} }. 
$$

First we use Proposition \ref{p:BWdecomp} to decompose 
$$
\mathcal{Q} = B_0^* \mathcal{Q}_+B_0 - B_{m+1}^* \mathcal{Q}_- B_{m+1}.
$$
Here, $\Phi:=(B_0,\mathbf{B}_+): \mathbb{R}^n \to E$ is bijective, where $E := \oplus_{i=0}^{m_+} \mathbb{R}^{n_i}$. Also, ${\rm ker}\, \mathbf{B}_+ \subseteq {\rm ker}\, B_{m+1}$, where $\mathcal{Q}_{\pm}$ are positive definite on $\mathbb{R}^{n_0}$ and $\mathbb{R}^{n_{m+1}}$ respectively. For reasons that will soon become apparent, we set
$$
Q_L:= \mathcal{Q}_+,\quad Q_R:= (\Phi^{-1})^* B_{m+1}^* \mathcal{Q}_- B_{m+1} \Phi^{-1}. 
$$

Clearly we have 
\begin{equation}\label{e:PropPhi}
	B_i\circ \Phi^{-1} = \pi_i,\quad i=0,1,\ldots,m_+
\end{equation}
and, with this in mind, we see that \eqref{e:NondegG} implies (in fact, is equivalent to)
\begin{align*}
	(\Phi^{-1})^*  \bigg( B_0^* \mathcal{Q}_+ B_0 - B_{m+1} \mathcal{Q}_- B_{m+1} + \sum_{i=1}^{m_+} c_i B_i^* G_i B_i \bigg) \Phi^{-1} >0
\end{align*}
and hence also 
\begin{align} \label{e:nondeg2}
	\pi_0^* Q_L  \pi_0 - Q_R + \sum_{i=1}^{m_+} c_i \pi_i^* G_i \pi_i >0. 
\end{align}
With the non-degeneracy condition \eqref{e:NondegFRBL-Nov12}  in mind, at this point we also observe that 
\begin{equation} \label{e:nondeg1}
\mathbb{R}^{n_0} \oplus\{0\}\oplus\cdots \oplus \{0\} \subseteq {\rm ker}\, Q_R 
\end{equation}
holds. To see this, it clearly suffices to check $B_{m+1}\Phi^{-1}(x_0,0,\ldots,0) = 0$. However,
\begin{align*}
\mathbf{B}_+\Phi^{-1}(x_0,0,\ldots,0) 
&= 
( B_1\Phi^{-1}(x_0,0,\ldots,0),\ldots, B_{m_+}\Phi^{-1}(x_0,0,\ldots,0) )\\
&= 
(\pi_1 (x_0,0,\ldots,0),\ldots, \pi_{m_+} (x_0,0,\ldots,0) ) =0
\end{align*}
and since ${\rm ker}\, \mathbf{B}_+ \subseteq {\rm ker}\, B_{m+1}$ we obtain the desired conclusion.

Next, given $f_i\in \mathcal{T}(G_i)$, $i = 1,\ldots,m_+$ and $h_j\in \mathcal{T}(G_j)$, $j = m_+ + 1,\ldots, m$, we introduce the function $h_{m + 1}:E \to \mathbb{R}_+$ by 
\begin{align*}
h_{m + 1}(x)
&:= 
e^{ -\pi \langle \pi_0x,\mathcal{Q}_+\pi_0x\rangle } \prod_{i=1}^{m_+} f_i ( B_i \circ \Phi^{-1}(x) )^{c_i}\\
&\quad \times 
e^{ \pi \langle \Phi^{-1}(x), B_{m+1}^* \mathcal{Q}_- B_{m+1} \Phi^{-1} (x) \rangle  }  \prod_{j=m_++1}^m h_j( B_j \circ \Phi^{-1}(x) )^{c_j}. 
\end{align*}
Setting $T_j:= B_j\circ \Phi^{-1}$ for $j = m_+ + 1,\ldots,m$,  and $T_{m + 1} = \mathrm{id}$, we then have
\begin{align*}
h_{m + 1}(T_{m + 1}x) &= e^{ - \pi \langle \pi_0x, Q_L \pi_0x\rangle} \prod_{i=1}^{m_+} f_i(\pi_ix)^{c_i} \times e^{\pi \langle x, Q_Rx\rangle}  \prod_{j=m_++1}^m h_j(T_jx)^{c_j}
\end{align*}
thanks to \eqref{e:PropPhi}. In particular, \eqref{e:AssumpFRBL} holds for inputs $(f_i)_{i = 1}^{m_+}$ and $(h_j)_{j = m_+ + 1}^{m + 1}$, and exponents $\mathbf{d} = ((c_i)_{i = 1}^{m_+}, (-c_{m_+ +1},\ldots, -c_m,1))$. Therefore, we may apply Theorem \ref{t:RegFRBL} with 
\[
\mathfrak{D} = ((T_j)_{j = m_+ + 1}^{m+1}, \mathbf{d},Q, ( (G_i)_{i = 1}^{m_+}, (G_{m_++1},\ldots,G_m,\infty)))
\]
to see 
\begin{equation}\label{e:ApplyFRBL}
\prod_{i=1}^{m_+} \bigg(\int_{\mathbb{R}^{n_i}} f_i \bigg)^{c_i}  \prod_{j=m_++1}^m \bigg(\int_{\mathbb{R}^{n_j}} h_j\bigg)^{c_j} 
\le 
{\rm FR}(\mathfrak{D}) \int_{E} h_{m + 1},
\end{equation}
where 
$$
{\rm FR}(\mathfrak{D}) = 
\sup 
\big( {\rm det}\, A \big)^\frac12 \prod_{i=1}^{m} \big( {\rm det}\, A_i \big)^{-\frac{c_i}2}
$$
and the supremum is taken over all $A_i \le G_i$, $i = 1,\ldots,m$, and $A>0$ satisfying 
\begin{equation}\label{e:SupCondition}
\pi_0^* Q_L \pi_0 + \sum_{i=1}^{m_+} c_i \pi_i^* A_i \pi_i \ge Q_R + \sum_{j=m_++1}^m (-c_j) T_j^* A_j T_j +A. 
\end{equation}
Here we remark that, strictly speaking, we are using a variant of Theorem \ref{t:RegFRBL} which admits $h_{m+1} \in \mathcal{T}(\infty)$. Such a result can be quickly obtained from Theorem \ref{t:RegFRBL} by a limiting argument, and this explains why the above matrix $A$ is an arbitrary positive definite matrix. The non-degeneracy condition is not affected by this limiting argument, and we have already verified it above in \eqref{e:nondeg2} and \eqref{e:nondeg1}.

Notice that $x_i = \pi_i (x_0,\ldots,x_{m_+}) = B_i \Phi^{-1}(x_0,\ldots,x_{m_+})$ from \eqref{e:PropPhi}. 
Hence, by the change of variable $\Phi^{-1}(x_0,\ldots,x_{m_+}) = y$, we see that 
\begin{align*}
	&\int_{E} h_{m+1}(x_0,\ldots,x_{m_+})\, dx_0\cdots dx_{m_+} \\
	&=  {\rm det} \, \Phi
	\int_{\mathbb{R}^n} e^{ -\pi \langle y, B_0^* \mathcal{Q}_+ B_0 y\rangle  } e^{\pi \langle y, B_{m+1}^* \mathcal{Q}_- B_{m+1} y\rangle } \prod_{i=1}^{m_+} f_i(B_i y)^{c_i} \prod_{j=m_++1}^m h_j(B_jy)^{c_j}  \, dy .
\end{align*}
Hence, it follows from \eqref{e:ApplyFRBL} that 
\begin{align*}
&\int_{\mathbb{R}^n} e^{-\pi\langle x,\mathcal{Q}x\rangle} \prod_{i=1}^{m_+} f_i(B_iy)^{c_i} \prod_{j=m_++1}^m h_j(B_jy)^{c_j} \, dy \\
&\ge {\rm FR}(\mathfrak{D})^{-1} {\rm det}\, \Phi^{-1} \prod_{i=1}^{m_+} \bigg( \int_{\mathbb{R}^{n_i}} f_i \bigg)^{c_i}
\prod_{j=m_++1}^{m} \bigg( \int_{\mathbb{R}^{n_j}} h_j \bigg)^{c_j}. 
\end{align*}
It remains to estimate ${\rm FR}(\mathfrak{D})$. To this end, we note that \eqref{e:SupCondition} is equivalent to 
$$
(\Phi^{-1})^* \bigg(
\mathcal{Q} + \sum_{i=1}^{m} c_i B_i^* A_i B_i \bigg)\Phi^{-1} \ge A
$$
and hence 
\begin{align*}
{\rm det}\, A 
&\le \frac1{({\rm det}\, \Phi)^2} 
{\rm det}\, \bigg( \mathcal{Q} + \sum_{i=1}^{m} c_i B_i^* A_i B_i \bigg).
\end{align*}
This shows that 
\begin{align*}
{\rm FR}(\mathfrak{D})
&\le 
\frac1{{\rm det}\, \Phi}  \sup_{\substack{A_i \leq G_i \\ i = 1,\ldots,m}  }
{\rm det}\, \bigg( \mathcal{Q}+ \sum_{i=1}^{m} c_i B_i^* A_i B_i  \bigg)^\frac12 \prod_{i=1}^{m} \big( {\rm det}\, A_i \big)^{-\frac{c_i}2} \\
&=
\frac1{ {\rm det}\, \Phi} \times  \frac1{ \mathrm{I}} 
\end{align*}
which concludes the proof of \eqref{e:RegIBL-Nov12}. 
\end{proof}

\begin{proof}[Theorem \ref{t:QMain} $\Rightarrow$ Theorem \ref{t:RegFRBL}]
We follow a limiting argument due to Wolff and presented in \cite[Section 4]{CouLiu}. First, let us show the following.
\begin{claim} \label{claim:Appendix}
Suppose $\mathfrak{D} = (\mathbf{T}, \mathbf{d}, Q,\mathbf{G})$ is a generalized forward-reverse Brascamp--Lieb datum which is non-degenerate (notation from Section \ref{subsection:RegFRBL} will prevail). Consider the Brascamp--Lieb datum $(\mathbf{B},\mathbf{c},\mathcal{Q})$ defined by 
\begin{align*}
\mathbf{B} & := ( \pi_1,\ldots,\pi_I,T_1,\ldots,T_J ), \\
\mathbf{c} & := (d_1,\ldots,d_I, -d(1),\ldots,-d(J)), \\ 
\mathcal{Q} & := \pi_0^* Q_L  \pi_0 -  Q_R.
\end{align*}
If $f_i \in \mathcal{T}(G_i),h_j\in \mathcal{T}(G(j))$ and $\widetilde{h} \in \mathcal{T}(\infty)$ satisfy 
\begin{equation}\label{e:AssumpFRBL-2}
	e^{ - \pi \langle \pi_0x, Q_L \pi_0x \rangle  } \prod_{ i=1 }^I  f_i(\pi_ix)^{d_i} \le e^{ - \pi \langle x,Q_R x\rangle } \prod_{j=1}^J h_j(T_jx)^{d(j)} \times \widetilde{h}(x),\quad x \in E,
\end{equation}
then 
\begin{equation}\label{e:App}
	{\rm I}(\mathfrak{D})
	\prod_{i=1}^I \bigg( \int_{E_i} f_i \bigg)^{d_i}
	\le  
	\prod_{j=1}^J \bigg( \int_{E(j)} h_j \bigg)^{d(j)}
	\int_{E} \widetilde{h},
\end{equation}
where
$$
{\rm I}(\mathfrak{D}) := \inf_{\substack{A_i\le G_i \\ A(j)\le G(j)}} {\rm BL} (\mathbf{B},\mathbf{c},\mathcal{Q};( (A_i)_{i=1}^I, (A(j))_{j=1}^J ).
$$  
\end{claim}
\begin{proof}[Proof of Claim \ref{claim:Appendix}]
In order to apply Theorem \ref{t:QMain}, let us check the non-degenerate conditions \eqref{e:NondegBW} and \eqref{e:NondegG}. It is clear from the setup that $\mathbf{B}_+ = (\pi_i)_{i=1}^I$ and thus the first condition in \eqref{e:NondegFRBL-Nov12} implies ${\rm ker}\, \mathbf{B}_+ \subseteq {\rm ker}\, Q_R$. This means $\mathcal{Q} x = \pi_0^* Q_L  \pi_0x$ whenever $x \in {\rm ker}\, \mathbf{B}_+$ which verifies the first condition in \eqref{e:NondegBW}. Also, we note that $s^+(\mathcal{Q}) \le n_0$ and hence $s^+(\mathcal{Q}) + \sum_{i=1}^I n_i \le \dim E$; this ensures the remaining condition in \eqref{e:NondegBW}. Finally, we observe that  \eqref{e:NondegG} is a direct consequence of the second condition in \eqref{e:NondegFRBL-Nov12}.  

Now \eqref{e:AssumpFRBL-2} clearly implies 
$$
\int_{E} \widetilde{h} 
\ge 
\int_{E} e^{-\pi \langle x,\mathcal{Q}x\rangle} \prod_{i=1}^I f_i(\pi_i x)^{d_i} \prod_{j=1}^J h_j (T_jx)^{-d(j)}\, dx  
$$
and an application of Theorem \ref{t:QMain} with the datum $(\mathbf{B},\mathbf{c},\mathcal{Q},\mathbf{G})$ yields \eqref{e:App}.
\end{proof}

Returning to the proof that Theorem \ref{t:QMain} implies Theorem \ref{t:RegFRBL}, we start by fixing a generalized forward-reverse Brascamp--Lieb datum $\mathfrak{D} = (\mathbf{T}, \mathbf{d}, Q,\mathbf{G})$ which is non-degenerate. 
For each $t > 0$, we set 
$$
d_{1,t}:= 1+ td_1,\ldots, d_{I,t}:= 1 + td_I,
\;d_{t}(1):= td(1),\ldots,d_{t}(J):=t d(J)
$$
and 
$$
Q_{L,t} := t Q_L,\; Q_{R,t} := t Q_R,  
$$
and consider the family of forward-reverse Brascamp--Lieb data $\mathfrak{D}_t = (\mathbf{T},\mathbf{d}_t, Q_t,\mathbf{G})$. Since the original data $\mathfrak{D}$ is non-degenerate, it is easy to verify that $\mathfrak{D}_t$ is non-degenerate for each $t > 0$.

Now fix $f_i \in \mathcal{T}(G_i), h_j\in \mathcal{T}(G(j))$ satisfying \eqref{e:AssumpFRBL} and set 
$$
\widetilde{h}_t(x):= e^{ - \pi \langle \pi_0x, Q_{L,t} \pi_0x\rangle } \prod_{i=1}^I  f_i(\pi_ix)^{d_{i,t}} e^{\pi \langle x, Q_{R,t} x\rangle} \prod_{j=1}^J h_j (T_jx)^{-d_{t}(j)}\quad  (x \in E).
$$
By Claim \ref{claim:Appendix},
$$
{\rm I}(\mathfrak{D}_t) \prod_{i=1}^I \bigg( \int_{E_i} f_i \bigg)^{d_{i,t}} \le \prod_{j=1}^J \bigg( \int_{E(j)} h_j \bigg)^{d_{t}(j)} \int_{E} \widetilde{h}_t,
$$
where 
$$
{\rm I}(\mathfrak{D}_t) := \inf_{\substack{ A_i\le G_i \\ A(j)\le G(j)}} {\rm BL} ( \mathbf{B},\mathbf{c}_t,\mathcal{Q}_t;( (A_i)_{i=1}^I, (A(j))_{j=1}^J )).
$$
Here, $\mathbf{B}$ is exactly as in Claim \ref{claim:Appendix}, and 
$$
\mathbf{c}_t:= (d_{i,t})_{i=1}^I,(-d_{t}(j))_{j=1}^J,\; \mathcal{Q}_t:= \pi_0^* Q_{L,t}  \pi_0 - Q_{R,t}. 
$$
Notice that \eqref{e:AssumpFRBL} and the definition of the exponents $d_{i,t},d_{t}(j)$ yield
\begin{align*}
\widetilde{h}_t(x) 
&= 	
\prod_{i=1}^I f_i(x_i)
\bigg(
e^{- \pi \langle x_0, Q_L x_0\rangle} 
\prod_{i=1}^I  
f_i(x_i)^{d_{i}}
e^{\pi \langle x,Q_{R} x\rangle }
\prod_{j=1}^J h_j(T_jx)^{-d(j)} 
\bigg)^t\\
&\le 
\prod_{i=1}^I f_i(x_i)
\end{align*}
and hence 
$$
{\rm I}(\mathfrak{D}_t) \prod_{i=1}^I \bigg( \int_{E_i}f_i \bigg)^{d_{i,t}} \le \prod_{j=1}^J \bigg( \int_{E(j)} h_j \bigg)^{d_{t}(j)} \prod_{i=1}^I \int_{E_i} f_i.
$$
After rearranging terms and taking a limit, we obtain 
$$
\prod_{i=1}^I \bigg( \int_{E_i} f_i \bigg)^{d_{i}} 
\le 
\liminf_{t\to \infty} {\rm I}(\mathfrak{D}_t)^{-\frac1t}
\prod_{j=1}^J \bigg( \int_{E(j)} h_j \bigg)^{d({j})}
$$ 
and so it suffices to check 
$$
\liminf_{t\to \infty}{\rm I}(\mathfrak{D}_t)^{-\frac1t} 
\le {\rm FR}(\mathfrak{D}).
$$

To investigate ${\rm I}(\mathfrak{D}_t)^{-\frac1t}$, we compute  
\begin{align*}
	&\int_{E} e^{ -\pi \langle x, \mathcal{Q}_tx\rangle } \prod_{i=1}^I \gamma_{A_i}(x_i)^{d_{i,t}} \prod_{j=1}^J \gamma_{A(j)}(T_j x)^{-d_{t}(j)}\, dx \\
	&= 
	\int_{E} 
	\exp\,\bigg( -\pi \bigg\langle x, \big( \pi_0^* Q_{L,t}  \pi_0 + \sum_{i=1}^I d_{i,t} \pi_i^* A_i \pi_i - Q_{R,t}  - \sum_{j=1}^J d_{t}(j) T_j^* A(j) T_j \big) x\bigg\rangle \bigg)\, dx\\
	&= 
	{\rm det}\, \bigg( t \pi_0^* Q_L \pi_0 + \sum_{i=1}^I (1+td_{i} )\pi_i^* A_i \pi_i - t Q_R  - \sum_{j=1}^J td({j}) T_j^* A(j) T_j \bigg)^{-\frac12}
\end{align*}
if $A_i,A(j)$ satisfy 
$$
t\pi_0^* Q_L  \pi_0 + \sum_{i=1}^I (1+td_{i} )\pi_i^* A_i \pi_i - t Q_R  - \sum_{j=1}^J t d({j}) T_j^* A(j) T_j
>0,
$$
otherwise the integral coincides with $+\infty$. 
Moreover, since we take $t\to\infty$, we may restrict attention to $A_i,A(j)$ satisfying 
$$
\pi_0^* Q_L \pi_0 + \sum_{i=1}^I d_{i}\pi_i^* A_i \pi_i - Q_R - \sum_{j=1}^J d({j}) T_j^* A(j) T_j 
\geq 0,
$$
which coincides with condition \eqref{e:GaussAssump}. With this in mind, and from the lower bound
\begin{align*}
	& {\rm BL} ( \mathbf{B},\mathbf{c}_t,\mathcal{Q}_t;( (A_i)_{i=1}^I, (A(j))_{j=1}^J )) \\
	&=\frac{\int_{E} e^{ -\pi \langle x, \mathcal{Q}_tx\rangle } \prod_{i=1}^I \gamma_{A_i}(x_i)^{d_{i,t}} \prod_{j=1}^J \gamma_{A(j)}(T_j x)^{-d_{t}(j)}\, dx}{\prod_{i=1}^I \big( \int_{E_i} \gamma_{A_i} \big)^{d_{i,t}} \prod_{j=1}^J \big( \int_{E(j)} \gamma_{A(j)} \big)^{-d_{t}(j)} }\\
	&\ge 
	{\rm det}\, \bigg( t \pi_0^* Q_L \pi_0 + \sum_{i=1}^I (1+td_{i} )\pi_i^* A_i \pi_i \bigg)^{-\frac12} \prod_{ i=1 }^I ({\rm det}\, A_i)^\frac{1+td_i}{2} \prod_{j=1}^J ( {\rm det}\, A(j) )^{ -\frac{td(j)}2  } \\
	&= 
	\bigg( t^{n_0} {\rm det}\, Q_L \prod_{i=1}^I(1+td_{i})^{n_i} 
	\prod_{i=1}^I ({\rm det}\, A_i)^{ -td_i} \prod_{j=1}^J ({\rm det}\, A(j))^{td(j)}  \bigg)^{-\frac12}
\end{align*}
we see that
\begin{align*}
&\liminf_{t\to \infty}{\rm I}(\mathfrak{D}_t)^{-\frac1t}\\
&\le 
\liminf_{t\to\infty} \bigg( t^{n_0} {\rm det}\, Q_L \prod_{i=1}^I (1+td_{i})^{n_i} \bigg)^{ \frac1{2t} }
\sup_{ \substack{A_i\le G_i,A(j)\le G(j):\\ \eqref{e:GaussAssump}} }\prod_{i=1}^I ({\rm det}\, A_i)^{ -\frac{d_i}2} \prod_{j=1}^J ({\rm det}\, A(j))^{\frac{d(j)}2}\\
&= 
\sup_{ \substack{ A_i\le G_i,A(j)\le G(j):\\ \eqref{e:GaussAssump}} }  \prod_{i=1}^I ({\rm det}\, A_i)^{ -\frac{d_i}2} \prod_{j=1}^J ({\rm det}\, A(j))^{\frac{d(j)}2}
\end{align*}
which concludes the proof. 
\end{proof}

\end{document}